\documentclass[11pt]{amsart}
\theoremstyle{plain}
\usepackage{amsmath, amsfonts, amsthm, amscd, mathrsfs}
\usepackage{geometry}
\usepackage{graphics}
\usepackage{enumerate}
\usepackage{color}
\usepackage{epsfig}
\usepackage[all]{xy}
\usepackage{graphicx,amssymb}
\usepackage{tikz}
\usepackage{bbm}

\graphicspath{ {Figures/} }
\numberwithin{figure}{section}
\usepackage{fullpage}

\newcommand{\HH}{\mathbb{H}}

\newtheorem{theorem}{Theorem}[section]
\newtheorem{lemma}[theorem]{Lemma}
\newtheorem{remark}[theorem]{Remark}
\newtheorem{proposition}[theorem]{Proposition}
\newtheorem{definition}[theorem]{Definition}
\newtheorem{corollary}[theorem]{Corollary}
\newtheorem{example}[theorem]{Example}
\newtheorem{claim}[theorem]{Claim}

\DeclareMathOperator{\Isom}{{\mathrm Isom}}
\DeclareMathOperator{\Hull}{{\mathrm Hull}}
\DeclareMathOperator{\stab}{{\mathrm stab}}

\DeclareMathOperator{\cusp}{{\mathrm cusp}}

\def\La{\Lambda}

\def\Ga{\Gamma}
\def\ga{\gamma}

\def\R{\mathbb R}

\def\N{\mathcal N}
\def\V{\mathcal{V}}
\def\A{\mathcal{A}}
\def\U{\mathcal U}
\def\bN{\mathbb N}

\def\mpar#1{}

\def\geo{\partial_{\infty}}

\newcommand{\Op}{O(\eta+e^{-\ell_{\gamma}/2})}

\title{Uniform spectral gap and orthogeodesic counting for strong convergence of Kleinian groups}
\author{Beibei Liu}
\address{Department of Mathematics, The Ohio State University, Columbus, OH 43210, USA}
\email{bbliumath@gmail.com}

\author{Franco Vargas Pallete}
\address{Department of Mathematics, Yale University, New Haven, CT 06511, USA}
\email{franco.vargaspallete@yale.edu}

\subjclass[2020]{}
\date{\today}
\thanks{}

\begin{document}
\maketitle
\begin{abstract}
   We show convergence of small eigenvalues for geometrically finite hyperbolic $n$-manifolds under strong limits. For a class of convergent convex sets in a strongly convergent sequence of Kleinian groups, we use the spectral gap of the limit manifold and the exponentially mixing property of the geodesic flow along the strongly convergent sequence to find asymptotically uniform counting formulas for the number of orthogeodesics between the convex sets. In particular, this provides asymptotically uniform counting formulas (with respect to length) for orthogeodesics between converging Margulis tubes, geodesic loops based at converging basepoints, and primitive closed geodesics.
\end{abstract}

\section{Introduction}

The critical exponent of a discrete isometry subgroup of the hyperbolic space $\HH^{n}$ is an important numerical invariant which relates the dynamical properties of the group action to the measure theory and the spectrum of operators on the quotient manifold via the  celebrated work of Patterson and Sullivan \cite{Patterson, Sullivan1, Sullivan2}. 
More explicitly, this invariant was shown to be equal to the Hausdorff dimension of the limit set for any geometrically finite discrete isometry subgroup $\Ga$ \cite{Sullivan1, Sullivan2}, and is related to the bottom spectrum $\lambda_{0}$  of the negative  Laplace operator for any nonelementary complete hyperbolic manifold \cite{Sullivan3}. 
A natural line of inquiry is to ask whether this quantitative invariant can be  uniformly controlled for a sequence of hyperbolic manifolds $(M_k=\HH^{n}/ \Ga_{k})_{k\in\mathbb{N}}$, for example, sequences of quasi-Fuchsian manifolds in Bers' model for the Teichm\"uller space of a surface $S$. It turns out that the critical exponent of $\Ga_{k}$,  the Hausdorff dimension of the limit set,  and the bottom of the spectrum $\lambda_{0}(\HH^{n}/ \Ga_{k})$, converge to the ones of the limit group $\Ga<\Isom(\HH^{n})$ under the assumption that $\Ga$ is geometrically finite and $\delta(\Ga)>(n-1)/2$  for strongly  convergent sequences of hyperbolic manifolds $(M_k)_{k\in \bN}$ \cite{CanaryTaylor, McMullen99}. See Section \ref{sec:convergenceconvex} for the definition of strong convergence. 
 
Besides the bottom spectrum of the quotient manifold, there are finitely many small eigenvalues of the negative Laplace operator in the interval $[\lambda_{0}, (n-1)^{2}/4]$, where  $(n-1)^{2}/4$ is the bottom spectrum of the hyperbolic space $\HH^{n}$ \cite{LaxPhillips}. It is natural to ask whether these small eigenvalues  converge to the ones of $\Ga$, respectively. We prove the convergence of small eigenvalues for strongly convergent sequences of hyperbolic manifolds $(M_k=\HH^{n}/ \Ga_{k})_{k\in\mathbb{N}}$. In particular, we give a uniform bound on the \emph{Lax-Phillips spectral gap} $s_1$ defined by $s_1:=\min\lbrace \lambda_1(M), (n-1)^2/4\rbrace - \lambda_0(M)$, where $\lambda_{1}(M)$ is the smallest eigenvalue of the negative Laplacian in $(\lambda_{0}(M), \infty)$. 
 
\begin{theorem}
\label{thm:uniformgap}
Suppose that $(M_k=\Isom(\HH^{n})/ \Ga_k)_{k\in\mathbb{N}}$ is a sequence of hyperbolic manifolds which converges strongly to a geometrically finite hyperbolic manifold $M=\HH^{n}/ \Ga$. The set of small eigenvalues in $[\lambda_{0}(M_k), (n-1)^{2}/4]$ converges to the small eigenvalues of the limit manifold $M$, counting multiplicities. In particular, the sequence of Lax-Phillips spectral gaps of $(M_k)_{k\in \bN}$ converges to that of the limit manifold $M$. 
\end{theorem}

\begin{remark}
We explain what the convergence of the set of small eigenvalues means in Section \ref{sec:eigenvalues}, and leave the precise statement in Theorem \ref{thm:spectralconvergence}. The statement of Theorem \ref{thm:uniformgap} for small eigenvalues holds for negatively pinched manifolds, and the details are discussed in Section \ref{sec:eigenvalues}. The statement referring to Lax-Phillips spectral gap is done in Theorem \ref{thm:Lax-Phillipsconvergence} for Kleinian groups. It could be possible that the set of small eigenvalues is equal to the singleton $\lbrace (n-1)^2/4\rbrace$ (or the empty set by considering pinched negative manifolds), but it won't affect the statement of the theorem. 

\end{remark}

Sequences of hyperbolic manifolds with uniform spectral gap are  interesting to study, as the uniform spectral gap sometimes controls the dynamical properties of the geodesic flow of the manifold. For instance, following \cite{EdwardsOh21},  uniform spectral gaps of hyperbolic manifolds imply uniform exponential mixings of geodesic flows. In the same paper, they provided another family of hyperbolic manifolds with uniform spectral gaps, coming from congruence subgroups of certain arithmetic lattice of $\Isom(\HH^{n})$.

The exponentially mixing geodesic flow can be used to find good estimates for error in asymptotic approximations of counting functions, such as the estimates available for orthogeodesic counting (as done in \cite{ParkkonenPaulin21}). Namely, given $D^-,D^+$ (locally) convex sets (or equivalently, $\pi_1(M)$ precisely invariant  convex sets in the universal covering) in $M$, one can estimate $\mathcal{N}_{D^-,D^+}(t)$, the number of orthogeodesics between $D^-$ and $D^+$ of length less than $t>0$, by
\[\mathcal{N}_{D^-,D^+}(t) \approx Ae^{\delta t}(1+O(e^{-\kappa t}))
\]
where $A,\delta,\kappa$ and $O(.)$ depend on the geometric/dynamical features of $M, D^-, D^+$, with exponential decay of correlations among these features. We consider the following two interesting cases in this paper:
\begin{enumerate}
\item $D^{\pm}$ are connected components in the thin part of $M$, i.e. Margulis tubes or cusps. 
\item $D^{+}=D^{-}$ is an embedded ball at a given point $x\in M$. That is, the lifts of $D^{\pm}$ are sufficiently small balls of lifts of $x$ in $\HH^{n}$. 

\end{enumerate}

The uniform orthogeodesic counting formula  for strongly convergent sequences in case (1) can be used in the study of  the renormalized volume. Given a hyperbolic manifold $M$,  the renormalized volume is a function on the deformation space of $M$ whose gradient flow has been of interest (see \cite{BBB19}, \cite{BBVP21}). In \cite{BBVP21} it is shown that for $M$ acylindrical the gradient flow of the renormalized volume converges to the unique critical point. This involves discarding strong limits with pinched rank-$1$ cusps by the use of the Gardiner formula. For such a method to work one needs a uniform control of contributing terms in the Gardiner formula, which would be provided by uniform orthogeodesic counting. The uniform orthogeodesic counting formula for case (2) gives a uniform asymptotic counting result with uniform error term  for geodesic loops based on a given point in $M$. 

Motivated by these applications, we show that the parameters $A,\delta,\kappa$ and $O(.)$ are uniform for strongly convergent  sequences, and such parameters can be taken arbitrarily close to the corresponding parameters of the geometrically finite limit.  

\begin{theorem}\label{thm:uniformcounting}
Let  $(M_{k}=\HH^{n}/ \Ga_{k})_{k\in\mathbb{N}}$ be a sequence of hyperbolic manifolds  which  strongly converges to a geometrically finite hyperbolic manifold $M=\HH^{n}/\Ga$ with $\delta(\Ga)>(n-1)/2$.
\begin{enumerate}
\item Suppose that $D^{\pm}_{k}$ are connected components in the thin part of $M_k$, and $(D^{\pm}_{k})_{k\in \bN}$ converge strongly to connected components $D^{\pm}$ in the thin part of $M$. Then there is a uniform counting formula for orthogeodesics between $D^{-}_{k}$ to $D^{+}_{k}$ for the sequence $(M_k)_{k\in\mathbb{N}}$. 
\item Suppose that $(x_k\in M_k)_{k\in\mathbb{N}}$ is a sequence of points converging to the point $x\in M$. Then there is a uniform counting formula for geodesic loops based at $x_k$ for the sequence $(M_k)_{k\in\mathbb{N}}$. 
\end{enumerate}
\end{theorem}

\begin{remark}
We in fact prove the result for  strongly convergent sequences of well-positioned convex sets in a strongly convergent sequence of hyperbolic manifolds (Theorem \ref{mainthm:counting}). We refer readers to Section \ref{sec:convergenceconvex} for the definitions of well-positioned and strong convergence of convex sets in hyperbolic manifolds. 
\end{remark}

The counting of primitive closed geodesics follows from the counting of geodesic loops in manifolds with negatively pinched curvatures \cite[Chapter 5]{Roblin03}. Hence we obtain the following asymptotic counting of primitive closed geodesics along sequences of strongly convergent hyperbolic manifolds.

\begin{corollary}
\label{coro:simple}
Suppose that $(M_{k}=\HH^{n}/ \Ga_{k})_{k\in\mathbb{N}}$ is a sequence of hyperbolic manifolds which strongly converges to a geometrically finite hyperbolic manifold $M=\HH^{n}/ \Ga$ with $\delta(\Ga)>(n-1)/2$. Then we can count the number of primitive closed geodesics with length less than $\ell$ in $M_k$, denoted by $\#\mathcal{G}_{M_k}(\ell)$,  uniformly, in the sense that
\[\#\mathcal{G}_{M_k}(\ell) \approx \frac{e^{\delta(\Gamma_k)\ell}}{\delta(\Gamma_k)\ell}
\]
up to a multiplicative error uniformly close to 1 along the sequence as $\ell$ gets larger and $\lim_k \delta(\Gamma_k) = \delta(\Gamma)$.
\end{corollary}

The proof of Theorem \ref{thm:uniformcounting} involves the uniformity of the exponential mixing and the convergence of certain measures for strongly convergent sequences. These measures refer to the classical Patterson-Sullivan measures, the Bowen-Margulis measure and the skinning measures. The convergence of Patterson-Sullivan measures has been proved for strongly convergent sequences under the assumption that the limit manifold is geometrically finite and its critical exponent is greater than $(n-1)/2$, \cite{McMullen99}. The Bowen-Margulis measure and the skinning measures are defined in terms of the Patterson-Sullivan measures. Answering an question of Oh, we  prove the convergence of these two measures, which could have its own interest. 

\begin{proposition}
\label{prop:margulismeasure}
Suppose that $(M_k=\HH^{n}/ \Ga_k)_{k\in\mathbb{N}}$ is a sequence of hyperbolic manifolds which are strongly convergent to a geometrically finite hyperbolic manifold $M=\HH^{n}/ \Ga$ with $\delta(\Ga)>(n-1)/2$. For $r>0$ we denote by $M_{k}^{<r}\subset M_k, M^{<r}\subset M$ the sets of points with injectivity radius less than r. Then the Bowen-Margulis measures $m_{\rm{BM}}^{k}$ on $T^{1}M_k^{<r}$ converge to the one on $T^{1}M^{<r}$ weakly. Moreover, we have the convergence of total masses. 

\end{proposition}

\begin{remark}
The convergence of the Bowen-Margulis measures on $T^{1}M_{k}^{<r}$ might be helpful for proving that  the Benjamini-Schramm limit of  $(M_k)_{k\in \bN}$ is also $M$ (see for instance \cite[Section 3.9]{7s} for a general definition of Benjamini-Schramm convergence).

\end{remark}

We now discuss the convergence of skinning measures $\sigma^{\pm}$ for the special type of \emph{well-positioned} convex sets in hyperbolic manifolds. Geodesic balls with sufficiently small radii and the thin part in a hyperbolic manifolds are well-positioned. We refer readers to Section \ref{sec:convergenceconvex} for the definition and detailed discussions. 

\begin{corollary}
\label{coro:convergenceskinning}
Suppose that  $(M_k=\HH^{n}/\Ga_k)_{k\in\mathbb{N}}$ is a sequence of hyperbolic manifolds that strongly converges to a geometrically finite hyperbolic manifold $M=\HH^{n}/ \Ga$ with $\delta(M)>(n-1)/2$.   Let $D_k\subset M_k, D\subset M$ be well-positioned convex sets, so that $(D_k)_{k\in \bN}$ strongly converges to $D$. Then
\[\Vert\sigma^{\pm}_{\partial D_k}\Vert \rightarrow \Vert\sigma^{\pm}_{\partial D}\Vert.\]
The relative result also holds for subsets $\Omega_k\subseteq D_k,\, \Omega \subseteq D$ so that $(\Omega_k)_{k\in \bN}$ strongly converges to $\Omega$. 
\end{corollary}

\medskip 
{\bf Organization of the paper.} We review  definitions of  geometric finiteness, the Bowen-Margulis measure, and skinning measures in Section \ref{subsec:geofinite}, \ref{subsec:PSmeasure}, \ref{subsec:BMmeasure}, respectively. Section \ref{subsec:spectrum} is about the relation between the critical exponent and the bottom spectrum. Section \ref{sec:convergenceconvex} defines strong convergence of hyperbolic manifolds and the convergence of well-positioned convex sets.  Section \ref{sec:eigenvalues} discusses small eigenvalues of the negative  Laplacian  on negatively pinched Hadamard manifolds and gives a proof of Theorem \ref{thm:uniformgap}. In Section \ref{sec:measure}, we prove  the convergence results  of the Bowen-Margulis measure and the skinning measures, i.e. Proposition \ref{prop:margulismeasure} and Corollary \ref{coro:convergenceskinning}. The last section, Section \ref{uniformcounting}, proves the uniform  asymptotic counting results of geodesic loops and orthogeodesics along strongly convergent sequences, i.e., the proof of  Theorem \ref{thm:uniformcounting}.

\section*{Acknowledgements} We would like to thank Martin Bridgeman for pointing out the uniform counting question to us, and Hee Oh for suggesting Propostion \ref{prop:margulismeasure} and useful discussions.  We are  very grateful to Curtis T. McMullen, Fr\'ed\'eric Paulin for helpful comments on an earlier draft, and Ian Biringer for email correspondence.  We  also appreciate the anonymous referees for the helpful suggestions.  The first author is partially supported by the NSF grant DMS-2203237. The second author is supported by NSF grant DMS-2001997.

\section{Background}\label{sec:background}
\subsection{Geometric finiteness}\label{subsec:geofinite}
In this subsection, we let $X$ denote an $n$-dimensional negatively pinched Hadamard manifold whose sectional curvatures lie between $-\kappa^{2}$ and $-1$ for some $\kappa\geq 1$. For any isometry $\ga\in \Isom(X)$, we define its \emph{translation length} $\tau(\ga)$ as follows:
$$\tau(\ga):=\inf_{p\in X} d_{X}(p, \ga(p)),$$
where $d_{X}$ is the Riemannian distance function in $X$. 
Based on the translation length, we can classify isometries in $X$ into 3 types;  we call $\ga$  \emph{loxodromic} if  $\tau(\ga)>0$. In this case, the infimum is attained exactly when the points are on the axis of $\ga$. The isometry $\ga$ is called \emph{parabolic} if $\tau(\ga)=0$ and the infimum is not attained. The isometry $\ga$ is \emph{elliptic} if $\tau(\ga)=0$ and the infimum is attained. 

From now on, we consider torsion-free discrete isometry subgroups $\Ga<\Isom(X)$, i.e. $\Ga$ contains no elliptic elements. If $\Gamma<\Isom(\HH^{n})$ is a torsion-free discrete isometry subgroup, we call it a \emph{Kleinian group}. 
Given $0<\epsilon<\epsilon(n, \kappa)$, where $\epsilon(n, \kappa)$ is the Margulis constant depending on the dimension $n$ and the constant $\kappa$, let $\mathcal{T}_{\epsilon}(\Ga)$ be the set consisting of all points $p\in X$ such that there exists an isometry $\ga\in \Ga$ with 
$$d(p, \ga p)\leq \epsilon. $$
It is an $\Ga$-invariant set, and the quotient $\mathcal{T}_{\epsilon}(\Ga)/ \Ga$ is the thin part of the quotient manifold $M=X/ \Ga$, denoted by $M^{<\epsilon}$.

  A subgroup $P<\Ga$ is called \emph{parabolic} if the fixed point set of $P$ consists of a single point $\xi\in \geo X$, where $\geo X$ is the visual boundary of $X$. Note that $\mathcal{T}_{\epsilon}(P)\subset X$ is precisely invariant under $P$, i.e.  $\stab_{\Ga}(\mathcal{T}_{\epsilon}(P))=P$ \cite[Corollary 3.5.6]{Bowditch95}. By abuse of notation, we can regard $\mathcal{T}_{\epsilon}(P)$ as a subset of $M=X/ \Ga$, which is called a \emph{Margulis cusp}. The union of all Margulis cusps consists of the \emph{cuspidal} part of $M$, denoted by $\cusp_{\epsilon}(M)$.

The \emph{limit set} $\Lambda(\Ga)$ of  a discrete, torsion-free isometry subgroup $\Ga<\Isom(X)$ is defined to be the set of accumulation points of a $\Ga$-orbit $\Ga(p)$ in $\geo X$ for any point $p\in  X$. We call $\Ga$ \emph{elementary} if $\La(\Ga)$ is finite; Otherwise, we say $\Ga$ is \emph{nonelementary}.  For any two points $\xi$ and $\eta$ in $\geo X$, we use $\xi\eta$ to denote the unique geodesic in $X$ connecting these two points. The \emph{convex hull} of $\La(\Ga)\subset \geo X$ is the smallest closed convex subset in $X$ whose accumulation set is $\La(\Ga)$, denoted by $\Hull(\Ga)$. We let $C(M)=\Hull(\Ga)/ \Ga$ denote the \emph{convex core} of quotient manifold $M=X/ \Ga$. For any constant $\epsilon>0$, we define the \emph{truncated core} by 
$$C(M)^{>\epsilon}=C(M)-M^{<\epsilon}. $$

Given a constant $0<\epsilon<\epsilon(n, \kappa)$, a discrete isometry subgroup $\Ga$ is \emph{geometrically finite} if the truncated core $C(M)^{>\epsilon}$ is compact in $M=X/ \Ga$. If, in addition, $C(M)$ is compact, i.e. $\Ga$ contains no parabolic isometries, then $\Ga$ is called \emph{convex co-compact}. Furthermore, if $\Ga<\Isom(X)$ is geometrically finite, the parabolic fixed points in $\Lambda(\Ga)$ are bounded, defined as follows:
\begin{definition}\cite{Bowditch93}
A parabolic fixed point $\xi\in \La(\Ga)$ is \emph{bounded} if $(\La(\Ga)\setminus \{p\})/ \stab_{\Ga}(p)$ is compact. 
\end{definition}

Given a point $x\in X$ and a discrete isometry group $\Ga\in \Isom(X)$, the Poincar\'e series is defined as 
$$P_{s}(\Gamma, x)=\sum_{\ga\in \Ga} e^{-sd_{X}(x, \gamma x)}. $$
The \emph{critical exponent} of $\Ga$ is defined as 
$$\delta(\Ga):=\inf\{s \mid P_{s}(\Gamma, x)<\infty\}. $$
It is not hard to see that the definition of $\delta(\Ga)$ is independent of the choice of $x$.

\subsection{Eigenvalues and spectrum}\label{subsec:spectrum}
As in Section \ref{subsec:geofinite}, we let $M=X/ \Gamma$, where $X$ is a negatively pinched Hadamard manifold, and $\Gamma$ is a torsion-free discrete isometry subgroup. 
Define the Sobolev space $H^1(M)$ as the space obtained by the completion of $C^\infty_0(M)$ with respect to the norm $\Vert f\Vert=\sqrt{\int_M|f|^2 + \int_M|\nabla f|^2}$. This space can be also defined as functions in $L^2(M)$ whose weak derivative (in the sense of distributions) is also in $L^2(M)$.

Given $f\in H^1(M)$, we define the Rayleigh quotient $R(f)$ of $f$ by
\[R(f) = \frac{\int_M|\nabla f|^2}{\int_M|f|^2}.
\]

The Rayleigh quotient is closely related to the spectrum $Spec(M)$ of the negative  Laplace operator. Namely, by posing the following minimization problem

\[\lambda=\inf \bigg\lbrace R(f)\,\bigg|\, f\in H^1(M) \bigg\rbrace
\]
we obtain a $L^2$ integrable smooth function $f$ satisfying $-\Delta f = \lambda f$.

We let $\lambda_{0}(M)$ denote the bottom of the spectrum, and we say that $\lambda\in Spec(M)$ is a small eigenvalue of $M$ if $\lambda<(n-1)^2/4$. Moreover, given a constant $\mu<(n-1)^2/4$, we define $Spec_\mu(M)$ as the collection (counting multiplicities) of  eigenvalues of the negative Laplacian on $M$ less than or equal to $\mu$. The set of small eigenvalues is a finite set (see \cite{Ursula04}).

In the rest of the subsection, we list several properties of the bottom of the spectrum $\lambda_{0}(M)$. We will use these properties in Section \ref{sec:eigenvalues} to prove the uniform spectral gap for strongly convergent sequences of geometrically finite groups $(\Ga_{k}<\Isom(X))_{k\in \bN}$.

\begin{lemma}\cite{Ursula04}\label{lemma:loxodromic}
Let $\Gamma<\Isom(X)$ be a torsion-free discrete elementary isometry subgroup of a negatively pinched Hadamard manifold $X$ with dimension $n$.  Then  $\lambda_0(X/\Gamma)\geq (n-1)^2/4$.
\end{lemma}

\begin{lemma}\cite[Lemma 2.3]{Ursula04}\label{lemma:expends}
Suppose that $\Gamma<\Isom(X)$ is a geometrically finite discrete isometry subgroup of a negatively pinched Hadamard manifold $X$ with dimension $n$. Then for every $r>0$ we have that $\mu_1(M\setminus B_r(C(M)))\geq (\tanh r)^2(n-1)^2/4$, where $M=X/ \Ga$ and  $\mu_1(M\setminus B_r(C(M)))$ denotes the smallest Rayleigh quotient for all smooth functions $f$ with compact support in $M\setminus B_r(C(M))$. 
\end{lemma}

If $X=\HH^{n}$, we have the following result relating $\lambda_{0}(M)$ to the critical exponent $\delta(\Ga)$. 

\begin{theorem}\cite{Sullivan3}
\label{Sullivan3}
For any nonelementary complete hyperbolic manifold $M=\HH^{n}/ \Ga$, one has

\begin{equation*}
\lambda_{0}(M)= \begin{cases}
(n-1)^{2}/4 &\text{if $\delta(\Gamma)\leq (n-1)/2$,}\\
\delta(\Gamma)(n-1-\delta(\Gamma)) &\text{if $\delta(\Gamma)\geq (n-1)/2$.}
\end{cases}
\end{equation*}

\end{theorem}

\subsection{Patterson-Sullivan measure}\label{subsec:PSmeasure}
Given a point $p\in \HH^{n}$, and $\xi\in \geo \HH^{n}$, the \emph{Busemann function} $B(x, \xi)$ on $\HH^{n}$ with respect to $p$ is defined by 
$$B(x, \xi)=\lim_{t\rightarrow \infty} (d(x, \rho_{\xi}(t))-t)$$
where $\rho_{\xi}(t)$ is the unique geodesic ray from $p$ to $\xi$. The \emph{Busemann cocycle} $\beta_{\xi}(x, y): \HH^{n}\times \HH^{n}\times \geo \HH^{n}\rightarrow \mathbb{R}$ is defined by 
$$\beta_{\xi}(x, y)=\lim_{t\rightarrow \infty} (d(\rho_{\xi}(t), x)-d(\rho_{\xi}(t), y)). $$ 

For a discrete isometry subgroup $\Ga<\Isom(\HH^{n})$, there exists a family of finite measures $(\mu_{x})_{x\in \HH^{n}}$ on $\geo \HH^{n}$ whose support is the limit set $\La(\Ga)$ and satisfies the following conditions:
\begin{enumerate}
\item It is $\Ga$-invariant, i.e. $\gamma_{\ast}(\mu_{x})=\mu_{\gamma x}$.
\item The Radon-Nikodym derivatives exist for all $x, y\in \HH^{n}$, and for all $\xi\in \geo \HH^{n}$ they satisfy 
$$\dfrac{d\mu_{x}}{d\mu_{y}}(\xi)=e^{-\delta(\Ga) \beta_{\xi}(x, y)}. $$

\end{enumerate}
Such family of measures is a family of \emph{Patterson-Sullivan density} of dimension $\delta(\Ga)$ for $\Ga$. The Patterson-Sullivan measures have very nice properties when the group $\Ga$ is geometrically finite.

\begin{theorem}\cite[Theorem 3.1]{McMullen99}
\label{PS1}
Let $\Ga<\Isom(\HH^{n})$ be a geometrically finite Kleinian group. Then $\geo\HH^{n}$ carries a unique $\Ga$-invariant density $\mu$ of dimension $\delta(\Ga)$ with total mass one; Moreover, $\mu$ is nonatomic and supported on $\La(\Ga)$, and the Poincar\'e series diverges at $\delta(\Ga)$. 

\end{theorem}

\begin{theorem}\cite[Theorem 1.2]{McMullen99}
\label{convergencePSmeasure}
Suppose that $(\Ga_{k}<\Isom(\HH^{n}))_{k\in \bN}$ is a sequence of Kleinian groups  converging strongly to $\Ga<\Isom(\HH^{n})$. If $\Ga$ is geometrically finite with $\delta(\Ga)>(n-1)/2$, then the Patterson-Sullivan densities $\mu_{k}$ of $\Ga_{k}$ converge to the Patterson-Sullivan density $\mu$ of $\Ga$ in the weak topology on measures. 
\end{theorem}

\begin{remark}
Theorem 1.2 in \cite{McMullen99} is stated for the $3$-dimensional hyperbolic space. However, the proof works exactly the same for general hyperbolic spaces $\HH^{n}$. 
\end{remark}

The proof Theorem \ref{convergencePSmeasure} relies heavily on the analysis of the Poincar\'e series of parabolic groups and its uniform convergence. This is also essential in the later proof of the convergence of Bowen-Margulis measures and the uniform counting formulas for orthogeodesics in the rest of the paper. For readers' convenience, we list the analytic properties of the Poincar\'e series corresponding to parabolic groups in the section. The details can be found in \cite[Section 6]{McMullen99}.

Let $L<\Isom(\HH^{n})$ be a torsion-free elementary isometry subgroup, which is either a hyperbolic group, i.e. a cyclic group generated by a loxodromic isometry,  or a parabolic group. Given $x\in \geo \HH^{n}$ and $s\geq 0$, the absolute Poincar\'e series for $L$ is defined to be 
$$P_{s}(L, x)=\sum_{\gamma\in L} |\gamma'(x)|^{s},$$
where the derivative is measured in the spherical measure. Given any open subset $U\subset \geo \HH^{n}$, define
$$P_{s}(L, U, x)=\sum_{\gamma(x)\in U} |\gamma'(x)|^{s}.$$
Suppose that $(L_k<\Isom(\HH^{n}))_{k\in\mathbb{N}}$ is a sequence of torsion-free elementary isometry subgroups which converges geometrically to a parabolic group $L<\Isom(\HH^{n})$ with parabolic fixed point $c$, i.e.,  $L_k$ converges to $L$ in the Hausdorff topology on closed subsets of $\Isom(\HH^{n})$.  The Poincar\'e series for  $(L_k, s_k)$ where $s_k\geq 0$ \emph{converges uniformly} if for any compact subset $K\subset \geo \HH^{n}\setminus \{c\}$ and $\epsilon>0$, there is a neighborhood $U$ of $c$ such  that for all $x\in K$, 
$$P_{s_k}(L_k, U, x)<\epsilon$$
for $k\gg 0$ sufficiently large. 
By using the same argument of the proof of Theorem 6.1 in \cite{McMullen99}, we have the following:
\begin{theorem}
\label{thm:cusptail}
Suppose that $(\Ga_{k}<\Isom(\HH^{n}))_{k\in\mathbb{N}}$ is a sequence of torsion-free discrete isometry subgroups which strongly converges to a geometrically finite torsion-free group $\Ga<\Isom(\HH^{n})$. Let $L<\Ga$ be a parabolic subgroup and $(L_k<\Ga_k)_{k\in\mathbb{N}}$ be a sequence of elementary groups which converges to $L$ geometrically. If 
\begin{equation*}
\delta(\Ga)> \begin{cases}
1 &\text{if $n=3$, or}\\
(n-2)/2 &\text{if $n> 3$,}
\end{cases}
\end{equation*}
then the Poincar\'e series for $(L_k, \delta(\Ga_k))$ converges uniformly to the one of $(L, \delta(\Ga))$.

\end{theorem}

\subsection{Bowen-Margulis measure}
\label{subsec:BMmeasure}
The \emph{Bowen-Margulis} measure is a measure defined on the unit tangent bundle $T^{1} \HH^{n}$ of $\HH^{n}$ in terms of the  Patterson-Sullivan measures. One can identify the unit tangent bundle $T^{1}\HH^{n}$ with the set of  geodesic lines $l: \R\rightarrow \HH^{n}$ such that the inverse map sends the geodesic line $l$ to its unit tangle vector $\dot{l}(0)$ at $t=0$. Given a  point $x_0\in \HH^{n}$,  we can also identify $T^{1}\HH^{n}$  with $\geo \HH^{n}\times \geo \HH^{n}\times \R$ via the \emph{Hopf's parametrization}:
$$v\rightarrow (v_{-}, v_{+}, t)$$
where $v_{-}, v_{+}$ are the endpoints at $-\infty$ and $\infty$ of the geodesic line defined by $v$ and $t$ is the signed distance of the closest point to $x_{0}$ on the geodesic line.

 We let $\pi: T^{1} \HH^{n}\rightarrow \HH^{n}$ denote the basepoint projection. The \emph{geodesic flow} on $T^{1}\HH^{n}$ is the smooth one-parameter group of diffeomorphisms $(g^{t})_{t\in \R}$ of $T^{1} \HH^{n}$ such that $g^{t}(l(s))=l(s+t)$, for all $l\in T^{1} \HH^{n}$, and $s, t\in \R$. Similarly one can define the geodesic flow on $T^{1}M$ by replacing the geodesic lines $l$ by locally geodesic lines. The Kleinian group $\Ga$ acts on $T^{1}\HH^{n}$ via postcomposition, i.e. $\gamma\circ l$, and it commutes with the geodesic flow. For simplicity, we sometimes write $\delta(\Ga)$ as $\delta$ if the context is clear in the rest of the paper. 

Given the Patterson-Sullivan density $(\mu_{x})_{x\in \HH^{n}}$ and a point $x_{0}\in \HH^{n}$, one can define the \emph{Bowen-Margulis measure} $\tilde{m}_{\rm{BM}}$ on $T^{1}\HH^{n}$ given by 
$$d\tilde{m}_{\rm{BM}}(v)=e^{-\delta(\beta_{v_{-}}(\pi(v), x_{0})+\beta_{v_{+}}(\pi(v), x_{0}))}d\mu_{x_{0}}(v_{-})d\mu_{x_{0}}(v_{+})dt= e^{-2\delta(v_-|v_+)_{x_0}}d\mu_{x_{0}}(v_{-})d\mu_{x_{0}}(v_{+})dt.$$
Here we introduce the notation $(v_-|v_+)_{x_0} = \frac12 (\beta_{v_{-}}(y, x_{0})+\beta_{v_{+}}(y, x_{0}))$, where $y$ is any point in the geodesic joining $v_-, v_+$. It is not hard to verify that $(v_-|v_+)_{x_0}$ does not depend on $y$.

The Bowen-Margulis measure $\tilde{m}_{\rm{BM}}$ is independent of the choice of $x_{0}$, and it is invariant under both the action of the group $\Ga$ and the geodesic flow. Hence, it descends to a measure $m_{\rm{BM}}$ on $T^{1}M$ invariant under the quotient geodesic flow, which is called the \emph{Bowen-Margulis} measure on $T^{1}M$.

\begin{theorem}\cite{Sullivan2, Babillot}
Let $\Ga<\Isom(\HH^{n})$ be a geometrically finite Kleinian group. The Bowen-Margulis measure $m_{\rm{BM}}$ has finite total mass, and the geodesic flow is mixing with respect to $m_{\rm{BM}}$. 

\end{theorem}

Another related measure we consider in the paper is the so called \emph{skinning measure}. Let $D$ be a nonempty proper closed convex subset in $\HH^{n}$. We denote its boundary by $\partial D$ and the set of points at infinity by $\geo D$. Let 
\begin{equation}
\label{definemap}
P_{D}: \HH^{n}\cup (\geo \HH^{n}\setminus \geo D)\rightarrow D
\end{equation}
be the \emph{closest point map}. In particular, for points $x\in \HH^{n}$, $P_{D}(x)$ is the point on $D$ which minimizes the distance function $d(y, x)$ for $y\in D$,  and for points $\xi\in \geo\HH^{n}\setminus \geo D$, $P_{D}(\xi)$ is the point  $y\in D$ which minimizes the function $y\rightarrow \beta_{\xi}(y, x_{0})$ for a given $x_{0}$.

The \emph{outer unit normal bundle} $\partial^{1}_{+}D$ of the boundary of $D$ is the topological submanifold of $T^{1}\HH^{n}$ consisting of the geodesic lines $v: \R\rightarrow \HH^{n}$ such that $P_{D}(v_{+})=v(0)$. Similarly, one can define the \emph{inner unit normal} bundle $\partial^{1}_{-} D$ which consists of geodesic lines $v$ such that $P_{D}(v_{-})=v(0)$. Note that when $D$ is totally geodesic, $\partial^{1}_{+}D=\partial^{1}_{-}D$. 
Given the Patterson-Sullivan density $(\mu_{x})_{x\in \HH^{n}}$, the \emph{outer skinning measure} on $\partial^{1}_{+}D$ is the measure $\tilde{\sigma}^{+}_{D}$  defined by 
$$d\tilde{\sigma}_{D}^{+}(v)=e^{-\delta\beta_{v_{+}}(P_{D}(v_{+}), x_{0})}d\mu_{x_{0}}(v_{+}).$$
Similarly, one can define the \emph{inner skinning measure} $\tilde{\sigma}_{D}^{-}$ on $\partial^{1}_{-}D$ as follows:
$$d\tilde{\sigma}_{D}^{-}(v)=e^{-\delta\beta_{v_{-}}(P_{D}(v_{-}), x_{0})}d\mu_{x_{0}}(v_{-}).$$

For simplicity, we sometimes  identify a precisely invariant subset  $C\subset M=\HH^{n}/ \Ga$ with its fundamental domain $\tilde{C}$ in the universal cover, and use the notation $\sigma^{\pm}_{\partial C}$ to denote the outer/inner skinning measure $\tilde{\sigma}^{\pm}_{\tilde{C}}$ on $\partial^{1}_{\pm}\tilde{C}$.

\subsection{Convergence of convex sets}
\label{sec:convergenceconvex}
In this subsection, we first define strong convergence that admits \textit{disconnected} limits. Suppose that $((M_k,g_k))_{k\in\mathbb{N}}$ is a sequence of  $n$-manifolds of pinched sectional curvature $-\kappa^2\leq K\leq -1$. We say that the sequence converges \emph{strongly} to a (possibly disconnected) geometrically finite $n$-manifold $(N=\cup_{i}^{m} N_i,g)$ if the following holds:
\begin{enumerate}
    \item There exist points $p_{k,i}\in M_k$, $p_{i}\in N_i$ so that $d(p_{k,i}, p_{k,j})\rightarrow+\infty$ for $i\neq j$ and $(M_k,p_{k,i})\rightarrow(N_i,p_i)$ geometrically, i.e., there exists an exhaustion $U_{1,i}\subset U_{2,i} \subset\ldots$ of relatively compact open sets of $(N_i,p_i)$ and smooth maps $\varphi_{k,i}:U_{k,i}\rightarrow M_k$  so that $\varphi_{k,i}(p_i) = p_{k,i}$ and $\varphi_{k,i}^*g_k$ converges smoothly in compact sets to $g$.
    \item For any $\epsilon$, the truncated cores $C(M_k)^{>\epsilon}$ converge to the disjoint union $\cup_i C(N_i)^{>\epsilon}$. This means that for large $k$  we have $C(M_k)^{>\epsilon} = \cup_{i=1}^m C(M_{k,i})^{>\epsilon}$ where $C(M_{k,i})^{>\epsilon}\subset Im(\varphi_{k,i})$ and $\varphi^{-1}_{k,i}(C(M_{k,i})^{>\epsilon})$ converges to $C(N_i)^{>\epsilon}$ in the Hausdorff topology of compact sets in $N_i$.
\end{enumerate}

The definition accommodates situations like Dehn drilling and pinching closed geodesics in hyperbolic 3-manifolds. The pinching case can result in disconnected limit manifolds. If $M_k$ and $N$ are hyperbolic manifolds, and $N$ is connected, this definition is equivalent to the one described in \cite{McMullen99} for strong convergence. Because of this, in the cases when $N$ is connected we will simply omit the mention of \emph{possibly disconnected}, as well as the sub-index $i$ from our notation.

Moreover, given a sequence  $(M_k)_{k\in\mathbb{N}}$ converging strongly to a possibly disconnected manifold $N$, we say that the sequence of functions $(f_k:M_k\rightarrow\mathbb{R})_{k\in \bN}$ converges strongly to a function $f:N\rightarrow\mathbb{R}$ if, with the notation above, we have that for any basepoint $p_i$ the sequence $(f_k\circ\varphi_{k,i})_{k\in \bN}$ converges smoothly in compact sets to $f$. Similarly, if $\Sigma_k,\Sigma$ are smooth properly embedded submanifolds in $M_k, N$, we say that $(\Sigma_k)_{k\in \bN}$ converges strongly to $\Sigma$ if $\varphi_{k,i}^{-1}(\Sigma_k)$ converges smoothly in compact sets to $\Sigma$.  Since for any fixed compact set in $N_i$ the maps $\varphi_{k, i}$ are embeddings for $k$ sufficiently large, we can define strong convergence of functions and submanifolds of $T^1 M_k$ to $T^1 M$ by composing the derivatives of $\varphi_{k, i}$ with the projections from $T^{*}M_k$ to $T^1 M_k$.

Using the definition of strong convergence, we obtain a straightforward corollary:
\begin{corollary}
Suppose that $(M_k)_{k\in\mathbb{N}}$ is a sequence of manifolds with negatively pinched curvature which converges strongly to a (possibly disconnected) geometrically finite manifold $N$ with negatively pinched curvature. Then the manifolds $M_k$ are also geometrically finite for sufficiently large $k$. 
\end{corollary}
\begin{proof}
Suppose that $N=\cup_{i}^{m} N_i$. The truncated core $C(N)^{>\epsilon}$ is compact for any $0<\epsilon<\epsilon(n, \kappa)$, since $N$ is geometrically finite. By item (2) in the definition of strong convergence, $C(M_{k})^{>\epsilon}$ is also compact for large $k$, since $C(M_{k, i})^{>\epsilon}$ is compact for large $k$, and all $1\leq i\leq m$. 
\end{proof}

In Section 3, we work on sequences of manifolds of  negatively pinched curvature that converge strongly to (possibly disconnected) limit manifolds. Given an $n$-dimensional  manifold $M$ with negatively pinched curvature (possibly disconnected) and a constant $\mu<(n-1)^{2}/4$, $Spec_{\mu}(M)$ is defined as the collection of eigenvalues of the negative Laplacian on $M$ less than $\mu$. If $M$ is disconnected,  $ Spec_\mu(M)$ agrees with the union of  $Spec_\mu$ of each component of $M$ (counting multiplicity). Specifically, a function $f:M\rightarrow\mathbb{R}$ satisfies the equation $-\Delta f = \lambda f$ if and only if its restriction to each component of $M$ is either an eigenfunction with eigenvalue $\lambda$, or $0$. Moreover, while taking orthonormal eigenfunctions for $M$ we can consider that each eigenfunction has support in a unique component of $M$.

In Section \ref{sec:measure} and Section \ref{uniformcounting}, we focus on sequences of hyperbolic manifolds strongly converging to connected limit manifolds. Suppose now $M=\HH^{n}/ \Ga$ is an $n$-dimensional hyperbolic manifold. As we stated in the Introduction, locally convex sets in $M$ are in 1-to-1 correspondence with $\Gamma$-precisely invariant convex sets in $\HH^{n}$ by the projection map $\textup{Proj}: \HH^{n}\rightarrow M$. In particular, we sometimes identify local convex sets with one of their lifts which are $\Gamma$-precisely invariant, and we don't consider immersed locally convex sets, e.g. nonprimitive closed geodesics. For simplicity, we will omit the word \emph{locally} and plainly denote the sets as convex.

We say that a convex set $D$ in $M$ is \emph{well-positioned} if $\partial \tilde{D}$ is smooth, where $\tilde{D}$ denotes the lift of $D$ to $\mathbb{H}^3$, and $\sigma^{\pm}_{\partial D}$ has compact support.

\begin{example}
Suppose that $M$ is a geometrically finite hyperbolic manifold. Embedded geodesic balls and the thin part of $M$ are well positioned convex sets. 
\end{example}

\begin{proof}
Geodesic balls with radii smaller than the injectivity radius of the center and Margulis tubes are compact convex subsets, so  they are well-positioned. The lifts of a cusp neighbourhood $D$ in $M$ are horoballs whose boundaries are smooth. Since $M$ is geometrically finite, all parabolic fixed points are bounded. Hence, the intersection of $\partial D$ with the convex core is compact. Thus, $\sigma^{\pm}_{\partial D}$ has compact support and $D$ is well-positioned. 

\end{proof}

Suppose that $(M_{k}=\HH^{n}/ \Ga_{k})_{k\in\mathbb{N}}$ is a sequence of hyperbolic manifolds that converges strongly to a geometrically finite hyperbolic manifold $M=\HH^{n}/ \Ga$. We say that well-positioned convex sets $D_k\subset M_k$ \emph{strongly converge} to a well-positioned convex set $D\subset M$ if
\begin{enumerate}
    \item the boundary $\partial D_k$ converges strongly to $\partial D$, or equivalently, the lifts of $\varphi_k^{-1}(\partial D_k)$ converge smoothly in compact sets to lifts of  $\partial D$, where $\varphi_k:U_k\rightarrow M_k$ are the smooth maps in the definition of strong convergence of $(M_k)_{k\in \bN}$,
    \item $\bar{\pi}(supp(\sigma^{\pm}_{\partial D_{k}}))$ is contained in $\varphi_k\left(N_1(\bar{\pi}(supp(\sigma^{\pm}_{\partial D})))\right)$ for large $k$, where $\bar{\pi}: T^{1}M\rightarrow M$ and $N_1$ denotes the 1-neighborhood.
 \end{enumerate}

\begin{example}
\label{ex:convergentset}
Let $(M_{k}=\HH^{n}/ \Ga_k)_{k\in\mathbb{N}}$ be a sequence of hyperbolic manifolds that converges strongly to a geometrically finite hyperbolic manifold $M=\HH^{n}/ \Ga$. 
\begin{enumerate}
    \item Suppose that $(x_{k})_{k\in\mathbb{N}}$ is a sequence of points in $M_{k}$ that converges to $x\in M$. The geodesic balls around the $x_{k}$ with radius $r$  converge strongly to the geodesic ball of $x$ with the same radius, where $r$ is smaller than the injectivity radius of $x$. 
    \item Given $0<\epsilon<\epsilon(n, \kappa)$, the thin parts $M_{k}^{<\epsilon}$ converge  strongly to  the thin part $M^{<\epsilon}$. 
\end{enumerate}
\end{example}

\section{Convergence of small eigenvalues}
\label{sec:eigenvalues}

In this section, we study the convergence of small eigenvalues and prove the uniform spectral gap for strongly convergent sequences of geometrically finite $n$-manifolds of negatively pinched curvature $-\kappa^2\leq K \leq -1$.

\begin{proposition}\label{prop:coremass}
Let $M$ be a geometrically finite Riemannian $n$-manifold of pinched sectional curvature $-\kappa^2\leq K\leq -1$ and let $\epsilon=\epsilon(n, \kappa)>0$ be the Margulis constant. Given  $\mu<(n-1)^2/4$, there exists a sufficiently large constant  $r(\mu)=r>0$ and  some constant $\eta(\mu,r)=\eta>0$, so that  if $f\in H^1(M)$ with $R(f)\leq\mu$, then $ \int_{B_{2r}(C(M)^{>\epsilon})} |f|^2\geq \eta\int_M |f|^2$. Moreover, one can take $\eta\rightarrow \bigg(1-\frac{4\mu}{(n-1)^2}\bigg)$ as $r\rightarrow + \infty$.
\end{proposition}

\begin{proof}
Since $C^\infty_0(M)$ is dense in $H^1(M)$, we can assume without loss of generality that $f$ is compactly supported with $\int_M f^2 = 1$.

Observe that we can find $C^1$ functions $g,h:\mathbb{R}\rightarrow[0,1]$ so that
\begin{itemize}
    \item $g^2(x) + h^2(x)=1$ for any $x\in\mathbb{R}$, 
    \item $supp(g)\subseteq (-\infty,1],\, supp(h)\subseteq (0,+\infty]$.
\end{itemize}

Given a positive constant $r>0$, we define  $u:=u_r(x),\alpha:=\alpha_r(x)\in C^1_0(M)$ satisfying the following properties, by using scalings of $g,h$ along equidistant sets to $\partial C(M)$:

\begin{enumerate}
    \item $0\leq u,\alpha\leq 1$
    \item $supp(u) \subseteq B_{2r}(C(M))$
    \item $supp(\alpha) \subseteq int(B^c_r(C(M)))$
    \item $u^2+\alpha^2\equiv 1$
    \item $|\nabla u|,|\nabla\alpha|\leq \frac{C}{r}$ everywhere in $M$, for some constant $C$ independent of $r$ and $M$.
\end{enumerate}

Similarly, for the thick-thin decomposition of $M$ we define functions $v:=v_r(x),\beta:=\beta_r(x)\in C^1_0(M)$ along equidistant sets to $\partial M^{>\epsilon}$ satisfying the following properties
\begin{enumerate}
    \item $0\leq v,\beta\leq 1$
    \item $supp(v) \subseteq B_{2r}(M^{>\epsilon})$
    \item $supp(\beta) \subseteq int(B^c_r(M^{>\epsilon}))$
    \item $v^2+\beta^2\equiv 1$
    \item $|\nabla v|,|\nabla\beta|\leq \frac{C}{r}$ everywhere in $M$, for some constant $C$ independent of $r$ and $M$.
\end{enumerate}

Define then $f_1:=uvf,\,f_2:=u\beta f,\,f_3:=\alpha f$ which are in $C^1_0(M)$. By the definitions of $u,\alpha,v,\beta$ we have

\begin{enumerate}
    \item $supp(f_1)\subseteq B_{2r}(M^{>\epsilon}\cap C(M))$
    \item $supp(f_2)\subseteq B^c_{r}(M^{>\epsilon})$
    \item $supp(f_3)\subseteq B^c_{r}(C(M))$
    \item $f^2_1+f^2_2+f^2_3=f^2$.
\end{enumerate}

We can expand $R(f_1)$ as

\[R(f_1) = \bigg(\int_M f^2|\nabla (uv)|^2+2uvf\langle\nabla (uv), \nabla f \rangle + u^2v^2|\nabla f|^2\bigg)\bigg/\bigg(\int_M u^2v^2f^2\bigg).
\]

Since $\nabla(uv) = u\nabla v + v\nabla u$ then it follows that $|\nabla(uv)|\leq \frac{2C}{r}$, and subsequently

\[ \int_M f^2|\nabla (uv)|^2 \leq \frac{4C^2}{r^2}.
\]

By Cauchy-Schwarz, we also have that

\[ \bigg|\int_M 2uvf\langle\nabla (uv), \nabla f \rangle\bigg| \leq 2\int_M |\langle f\nabla(uv),\nabla f \rangle| \leq \frac{4C}{r}\sqrt{R(f)}.
\]

Collecting these inequalities and defining $a:=\int_M f^2_1$ for convenience, we arrive to

\begin{equation}\label{eq:Rf1bound}
    R(f_1) \leq \bigg( \frac{4C^2}{r^2} + \frac{4C}{r}\sqrt{R(f)} + \int_M u^2v^2|\nabla f|^2\bigg)\bigg/a.
\end{equation}

Similarly, define $b:=\int_M f^2_2,\,c:=\int_M f^2_3$. Then

\begin{equation}\label{eq:Rf2bound}
    R(f_2) \leq \bigg( \frac{4C^2}{r^2} + \frac{4C}{r}\sqrt{R(f)} + \int_M u^2\beta^2|\nabla f|^2\bigg)\bigg/b,
\end{equation}

\begin{equation}\label{eq:Rf3bound}
    R(f_3) \leq \bigg( \frac{4C^2}{r^2} + \frac{4C}{r}\sqrt{R(f)} + \int_M \alpha^2|\nabla f|^2\bigg)\bigg/c.
\end{equation}

By doing $a(\ref{eq:Rf1bound})+b(\ref{eq:Rf2bound})+c(\ref{eq:Rf3bound})$ we obtain

\begin{equation}
\begin{aligned}
    aR(f_1)+bR(f_2)+cR(f_3) &\leq \bigg( \frac{12C^2}{r^2} + \frac{12C}{r}\sqrt{R(f)} + \int_M (u^2v^2 + u^2\beta^2 + \alpha^2)|\nabla f|^2\bigg)\\
    &= \bigg( \frac{12C^2}{r^2} + \frac{12C}{r}\sqrt{R(f)} + \int_M |\nabla f|^2\bigg)\\
    &\leq\bigg( \frac{12C^2}{r^2} + \frac{12C}{r}\sqrt{\mu} + \mu \bigg).
\end{aligned}
\end{equation}
   
Since $supp(f_2)\subseteq int(B^c_{r}(M^{>\epsilon})),\,supp(f_3)\subseteq int(B^c_{r}(C(M)))$ we have by Lemmas \ref{lemma:loxodromic} and  \ref{lemma:expends} (or more precisely, by applying a combination of  the Lemmas on each component of $M$) that $R(f_2)\geq (\tanh r)^2(n-1)^2/4,\,R(f_3)\geq  (\tanh r)^2(n-1)^2/4$. Using these bounds together with the obvious bound $aR(f_1)\geq 0$ we arrive to

\[(b+c)\frac{(\tanh r)^2(n-1)^2}{4} \leq  \frac{12C^2}{r^2} + \frac{12C}{r}\sqrt{\mu} + \mu,
\]

\[(b+c) \leq \frac{4}{(\tanh r)^2(n-1)^2}\bigg( \frac{12C^2}{r^2} + \frac{12C}{r}\sqrt{\mu} + \mu\bigg).
\]

By the fact  that $a+b+c=1$ we obtain

\[a \geq \frac{4}{(\tanh r)^2(n-1)^2}\bigg(\frac{(\tanh r)^2(n-1)^2}{4} - \frac{12C^2}{r^2} - \frac{12C}{r}\sqrt{\mu} - \mu\bigg).
\]

The result follows from observing that for the left-hand side we have $\int_{B_{2r}(M^{>\epsilon}\cap C(M))} f^2 \geq a$, whereas the right hand side depends only on $\mu,r$ and converges to $1-\frac{4\mu}{(n-1)^2}>0$ as $r\rightarrow+\infty$.

\end{proof}

Now we use Proposition \ref{prop:coremass} to take limits of eigenfunctions with small eigenvalues along a strongly convergent sequence of manifolds with negatively pinched curvature.

\begin{lemma}\label{lem:eigensubsequence}
Suppose that $(M_k)_{k\in\mathbb{N}}$ is a sequence of  $n$-manifolds of pinched curvature $-\kappa^{2}\leq K\leq -1$ that converges strongly to a (possibly disconnected) geometrically finite $n$-manifold $M$.  Let $\mu<(n-1)^2/4$ and for each $M_k$, let $f_k$ be an eigenfunction of the negative Laplacian so that $R(f_k)\leq \mu$ and  $\int_{M_k}|f_k|^2 = 1$. Then, after possibly taking a subsequence, we have that $f_k$ converges strongly to $f$, a non-zero eigenfunction of the negative  Laplacian in $M$ with $R(f)\leq \mu$.
\end{lemma}

\begin{proof}
By Proposition \ref{prop:coremass} there exist $r>0$ and $\eta>0$ independent of $k$ so that $\int_{B_{2r}(C(M_k)^{>\epsilon})} |f_k|^2\geq \eta$. By elliptic regularity and strong convergence, we have that the Sobolev norms

\[\Vert f_k \Vert_{W^{2,\ell}(B_{2r}(C(M_k)^{>\epsilon}))}
\]
are uniformly bounded for any given $\ell$. By the Rellich-Kondrachov compactness theorem, we can take a convergent subsequence with limit $f$ in $B_{2r}(C(M)^{>\epsilon})$ in any $W^{2,\ell}$ norm. Taking $r\rightarrow+\infty$ and doing a Cantor diagonal argument, we have that $R(f) \leq \mu,\, -\Delta_M f = R(f)f,\, \int_{M} |f|^2\geq \eta$, which concludes the Lemma.
\end{proof}

Recall that $Spec_{\mu}(M)$ denotes  the collection of eigenvalues of the negative  Laplacian  on the negatively pinched manifold $M$ which are smaller than $\mu$, where for convenience we assume that $\mu<(n-1)^2/4$ is not an eigenvalue of $M$ (this is possible for all $\mu<(n-1)^2$/4 with the exception of finitely many values). Suppose that $(M_{k})_{k\in\mathbb{N}}$ is a sequence of negatively pinched manifolds  which converges strongly to a geometrically finite $n$-manifold $M$. Given any small eigenvalue $\lambda\in Spec_\mu(M)$, we can use the discreteness of small eigenvalues to take $\epsilon>0$ small enough so that $(\lambda-\epsilon,\lambda+\epsilon)\cap Spec_\mu(M) = \lbrace \lambda \rbrace$. We have then that $(\lambda-\epsilon,\lambda+\epsilon)\cap Spec_\mu(M_k)$ is either empty or accumulates to $\lambda$ as $k\rightarrow\infty$, where we desire to prove the later case. Let then $m_\lambda$ be the multiplicity of $\lambda$ and $m_{\lambda,k}$ be the cardinality of $(\lambda-\epsilon,\lambda+\epsilon)\cap Spec_\mu(M_k)$ (counting multiplicities). We say $Spec_{\mu}(M_k)$ \emph{converges} to $Spec_{\mu}(M)$, if $\lim_{k\rightarrow \infty} m_{\lambda,k} = m_\lambda$  for any small eigenvalue $\lambda\in Spec_{\mu}(M)$.

\begin{theorem}\label{thm:spectralconvergence}
Suppose that $(M_k)_{k\in\mathbb{N}}$ is a sequence of  $n$-manifolds of pinched curvature $-\kappa^{2}\leq K\leq -1$ that converges strongly to a (possibly disconnected) geometrically finite $n$-manifold $M$. Then for any given $\mu<(n-1)^2/4$ not in $Spec(M)$ we have that $Spec_\mu(M_k)$ converges (counting multiplicities) to $Spec_\mu(M)$.
\end{theorem}
\begin{proof}
To prove the theorem, we will show the convergence of eigenspaces. Namely, let $V_k, V$ denote the linear spaces of functions generated by the eigenfunctions with eigenvalues in $Spec_\mu(M_k)$ and $ Spec_\mu(M)$, which   have a natural orthogonal decomposition by the eigenspaces of $Spec_\mu(M_k)$ and $Spec_\mu(M)$. We  show that $V_k\rightarrow V$, in the following  sense:
\begin{enumerate}
\item Any function $f\in V$ can be obtained as the limit of a strongly convergent sequence $(f_k\in V_k)_{k\in\mathbb{N}}$.
\item Any sequence of families $( f_{l,k} \subset V_k)_{k\in\mathbb{N}}$ of orthonormal functions in $M_k$ converges strongly (after possibly taking a subsequence) to a linearly independent family of functions in $M$.

\end{enumerate}

Item (1) implies that $\liminf_{k\rightarrow \infty} m_{\lambda, k}\geq m_\lambda$, and Item (2) implies that $\limsup_{k\rightarrow \infty} m_{\lambda, k}\leq m_\lambda$. Thus, it suffices to prove the convergence of eigenspaces.  We first show item (2). Suppose that $f_{1,k},\ldots f_{l,k}$ are orthonormal eigenfunctions of $M_k$. By Lemma \ref{lem:eigensubsequence} we can assume they converge in compact sets to $f_1,\ldots,f_l$. If  the functions $f_1,\ldots, f_l$ are not linearly independent in $L^2(M)$, there exist real numbers $\alpha_1,\ldots,\alpha_l$ not all vanishing so that $\alpha_1f_1+\ldots+\alpha_lf_l\equiv 0$. Hence, $g_k = \alpha_1f_{1,k}+\ldots+\alpha_lf_{l, k}$ are functions in $H^1(M_k)$ with norm $\sqrt{\alpha_1^2+\ldots + \alpha_k^2}\neq0$. We can normalize $\Vert g_k\Vert_{L^2(M)}=1$ so that   $R(g_k)\leq \mu$, and since the limit of $g_k$ in compact sets is not identically zero from Proposition \ref{prop:coremass}, we have a contradiction.
    
    Now we prove Item (1).  Assume  that not all functions in $V$ are obtained as limits of functions in $V_k$. Let $V'$ be the proper maximal space in $V$, consisting of functions that can be obtained as limits. Assume that there exists an eigenfunction $f$ of $M$ with eigenvalue $\lambda$, such that $f$ is orthogonal to $V'$. Approximate $f$ in $H^1(M)$ by a compactly supported function $f_0$, which is normalized so that $\int_M|f_0|^2=1$ and $R(f_0)$ is close to $\lambda$. It follows  that $\int_M f_0g =\int_M (f_0-f)g$ can be taken uniformly small for all $g\in V'$ with $\int_M |g|^2=1$. Let $f^{k}_{0}$ be the pullback of $f_{0}$ in $M_k$ by the maps $\varphi_{k,i}$ from the definition of strong convergence. Then for sufficiently large $k$ we have that (after identifying the compact cores) $\int_{M_k} f_{0}^{k}g_k$ can be also taken uniformly small for any $g_k\in V_k$ with $\int_{M_k}|g_k|^2=1$ by Proposition \ref{prop:coremass}. For large $k$ we also have that in $M_k$ the Rayleigh quotient $R(f^{k}_{0})$ is close to $\lambda$. Denote then by $f_{0,k}$ the projection of $f_{0}^{k}$ perpendicular to $V_k$. Then $R(f_{0, k})$ is also very close to $\lambda$ for sufficiently large $k$. Hence, this contributes to an eigenfunction in $M_k$ which does not belong to $V_k$. However, by construction, $V_k$ is the linear space of functions generated by eigenfunctions with eigenvalues in $Spec_{\mu}(M_k)$, which gives a contradiction. Therefore, any function $f\in V$ can be obtained as the limit of a strongly convergent sequence $(f_k\in V_k)$.

\end{proof}

Recall that the Lax-Phillips spectral gap $s_{1}=\min\{\lambda_1(M), (n-1)^{2}/4\}-\lambda_{0}(M)$ for a hyperbolic manifold $M=\HH^{n}/ \Ga$. We obtain the following convergence result of spectral gap for strongly convergent sequences of hyperbolic manifolds.

\begin{theorem}\label{thm:Lax-Phillipsconvergence}
Suppose that $(M_k=\Isom(\HH^{n})/ \Ga_k)_{k\in\mathbb{N}}$ is a sequence of hyperbolic manifolds which converges strongly to a geometrically finite hyperbolic manifold $M=\HH^{n}/ \Ga$. Then the sequence of Lax-Phillips spectral gaps $s_1(M_k)$ converges to $s_1(M)$. 
\end{theorem}
\begin{proof}
By \cite[Theorem 1.5]{McMullen99} and Theorem \ref{Sullivan3} we have that $\lim_{k\rightarrow\infty} \lambda_0(M_k) = \lambda_0(M)$. By Theorem \ref{thm:spectralconvergence}, if $\lambda_{1}(M)\geq(n-1)^{2}/4$, then $\liminf\lambda_{1}(M_k)\geq(n-1)^{2}/4$ for sufficiently large $k$, or if $\lambda_{1}(M)<(n-1)^{2}/4$, we have that $\lim_{k\rightarrow\infty} \lambda_1(M_k) = \lambda_1(M)>\lambda_0(M)$ for sufficiently large $k$. In either case, the convergence of $s_1(M_k)$ to $s_1(M)$ follows.
\end{proof}
\medskip
\noindent
{\bf Proof of Theorem \ref{thm:uniformgap}: }
The proof follows  from Theorem \ref{thm:spectralconvergence} and Theorem \ref{thm:Lax-Phillipsconvergence}.

\section{Uniform convergence of measures}
\label{sec:measure}
In this section, we prove convergence for skinning measures and the Bowen-Margulis measure under strong convergence.  We assume that $M$ is a hyperbolic $n$-manifold, and by $M^{<\epsilon}$ we denote the $\epsilon$-thin part of $M$ for a constant $\epsilon$ smaller than the $n$-dimensional Margulis constant. We first prove that the Bowen-Margulis measure of the thin part is (uniformly) relatively small. 
 
\begin{proposition}\label{lem:ControlThinPart}
Suppose $(M_k=\HH^{n}/ \Ga_k)_{k\in\mathbb{N}}$ is a sequence of hyperbolic manifolds that strongly converges to a geometrically finite hyperbolic manifold  $M=\HH^{n}/ \Ga$ with $\delta(\Ga)>(n-1)/2$. Let $m^k_{\rm BM}, m_{\rm BM}$ be the Bowen-Margulis measures on $T^{1}M_{k}$ and $T^{1}M$, respectively. Then for any $\alpha>0$ there exist $\epsilon>0$ and $N>0$ so that for $\epsilon'<\epsilon$ and $k>N$ we have that
\[\int_{T^{1}M^{<\epsilon'}_{k}} dm^{k}_{\rm BM} <\alpha.
\]

\end{proposition} 

\begin{proof}
This follows Dalbo-Otal-Peigne's proof  \cite{DalboOtalPeigne00} on the finiteness of $m_{\rm BM}$.  We first let $\epsilon>0$ be a constant which is smaller than the shortest geodesic in $M$. Take a fundamental domain $F$ for the convex core of $M$ in the universal cover $\HH^{n}$, and divide $F$ as the thin part $F^{<\epsilon}$ (i.e., the intersection of $F$ with the thin part of $M$) and the thick part $F^{>\epsilon}$.  Consider  a component $D$ of $F^{<\epsilon}$, which must be a cuspidal component. Suppose that  $\mathcal{H}$ is the corresponding horoball based at the parabolic fixed point $\xi$, so that $D$ is a fundamental domain for the parabolic subgroup $\mathcal{P}<\pi_1(M)$ that preserves $\mathcal{H}$.

As detailed in \cite[page 118]{DalboOtalPeigne00} we can bound $\tilde{m}_{\rm BM}$ in $D$ by

\[\tilde{m}_{\rm BM}(T^1D) \leq \sum_{p\in\mathcal{P}}\int_{\mathcal{D}\times p\mathcal{D}} c^\mu(d\eta^-d\eta^+) \int_{(\eta^-\eta^+)\cap\mathcal{H}}dt,
\]
where  $c^\mu(d\eta^-d\eta^+)=e^{-2\delta(\Gamma)(\eta^-|\eta^+)_{x}}d\mu_{x}(\eta^{-})d\mu_{x}(\eta^{+})$ for a given point $x\in \HH^{n}$ and  $\mathcal{D}\subseteq \geo \HH^{n}\setminus\lbrace\xi\rbrace$ is a compact set such that $\lbrace p\mathcal{D} \rbrace_{p\in\mathcal{P}}$ covers $\Lambda(M) \setminus\lbrace \xi\rbrace$.  The existence of the compact set $\mathcal{D}$ is ensured by the assumption that $M$ is geometrically finite, hence the parabolic fixed point $\xi$ is bounded \cite{Bowditch93}. Now, let $\mathcal{P}_k$ be the elementary group in $M_k$ that converges to $\mathcal{P}$, which is either parabolic or loxodromic. We discuss the proof that when the groups $\mathcal{P}_{k}$ are loxodromic. The argument for parabolic subgroups is similar. 

Let  $\mathcal{H}_k$ be a neighborhood of the geodesic $\xi^-_k\xi^+_k$ preserved by $\mathcal{P}_k$ so that $\mathcal{H}_k\rightarrow\mathcal{H}$, $\xi^\pm_k\rightarrow\xi$. Since $M_k$ converges to $M$ strongly, by \cite{McMullen99} we can take $\mathcal{D}$ large enough so that $\lbrace p\mathcal{D} \rbrace_{p\in\mathcal{P}_k}$ covers $\Lambda(M_k)\setminus\lbrace \xi^-_k, \xi^+_k\rbrace$. Hence it follows that

\[m^k_{\rm BM}(T^1(\mathcal{H}_k/\mathcal{P}_k)) \leq \sum_{p\in\mathcal{P}_k}\int_{\mathcal{D}\times p\mathcal{D}} c^\mu_k(d\eta^-d\eta^+) \int_{(\eta^-\eta^+)\cap\mathcal{H}_k}dt.
\]

Assume  without loss of generality that we can take a common point $x\in\mathcal{H}_k,\mathcal{H}$. There exist compact set $K\subset \mathbb{H}^n$ and open neighbourhood $V\subseteq \mathbb{H}^n$ of $\xi$ so that for $k$ large,  if the (oriented) geodesic $\eta^-\eta^+$ with $\eta^-\in\mathcal{D}\cap\Lambda(\Gamma_k)$ and $\eta^+\in p\mathcal{D}\cap \Lambda(\Gamma_k)$  intersects $\mathcal{H}_k$, then the point of entry belongs to $K\cap \partial\mathcal{H}_k$ and $p^{-1}x$ belongs to $V$. In particular such geodesic $\eta^-\eta^+$ verifies $0\leq (\eta^-|\eta^+)_x\leq diam(K)$. Moreover,  we have that $|\int_{(\eta^-\eta^+)\cap\mathcal{H}_{k}}dt - d(x,px)|<2diam(K)$. Hence there exists a constant $C>0$ depending only on $diam(K)$ so that

\[m^k_{\rm BM}(T^1(\mathcal{H}_k/\mathcal{P}_k)) \leq C\left(\sum_{p\in\mathcal{P}'_k} \mu^k_x(\mathcal{D})\mu^k_x(p\mathcal{D})(d(x,px)+C)\right),
\]
where $\mu^{k}_{x}$ denotes the Patterson-Sullivan measure on $M_k$ and $\mathcal{P}'_k$ is the subset of $\lbrace p\in\mathcal{P}_k\,|\, p^{-1}x \in V \rbrace$ so that the summand $\int_{\mathcal{D}\times p\mathcal{D}} c^\mu_k(d\eta^-d\eta^+) \int_{(\eta^-\eta^+)\cap\mathcal{H}_k}dt$ is non-zero.

Recall that 

\[\mu^k_x(p\mathcal{D}) = \int_{\mathcal{D}} e^{-\delta(\Ga_k)B_\eta(p^{-1}x,x)}\mu^k_x(d\eta),
\]
so we would like to estimate $B_\eta(p^{-1}x,x)$. Observe that as $\mathcal{H}_k$ is preserve by $\mathcal{P}_k$, we have that if $\eta^-\eta^+$ is a geodesic with $\eta^-\in\mathcal{D}\cap\Lambda(\Gamma_k)$ and $\eta^+\in p\mathcal{D}\cap \Lambda(\Gamma_k)$ that intersects $\mathcal{H}_k$, then the exit point of $\eta^-\eta^+$ from $\mathcal{H}_k$ belongs to $pK\cap\partial\mathcal{H}_k$. By triangular inequality we have that under such conditions $|\int_{(\eta^-\eta^+)\cap\mathcal{H}_{k}}dt - B_\eta(x,px)|<2diam(K)$. Hence for $p\in\mathcal{P}'_k$
have $|B_\eta(p^{-1}x,x) - d(p^{-1}x,x)|\leq 4diam(K)$. Combining this with our previous inequality (and making the domain of the sum bigger if necessary) we get

\[m^k_{\rm BM}(T^1(\mathcal{H}_k/\mathcal{P}_k)) \leq C\left(\sum_{p\in\mathcal{P}_k, p^{-1}x\in V} (\mu^k_x(\mathcal{D}))^2e^{-\delta(\Ga_k)d(p^{-1}x,x)}(d(x,px)+C)\right).
\]
for some $C>0$ independent of $\epsilon$ and $k$.

We claim  that the above discussion holds for  smaller $\epsilon$ corresponding to a smaller neighborhood $V(\epsilon)$ for the same basepoint $x$. Consider a smaller thin part corresponding to  $\epsilon'<\epsilon$. The sets $\mathcal{H}_k, K$ vary with $\epsilon'$, although it is clear that $\mathcal{H}_k(\epsilon')\subset\mathcal{H}_k(\epsilon)$ and $diam(K(\epsilon')) < diam(K(\epsilon))$. Hence after taking a basepoint $y\in K(\epsilon')$ we have

\[m^k_{\rm BM}(T^1(\mathcal{H}_k(\epsilon')/\mathcal{P}_k)) \leq  C\left(\sum_{p\in\mathcal{P}_k, p^{-1}y\in V(\epsilon')} (\mu^k_y(\mathcal{D}))^2e^{-\delta(\Ga_k)d(p^{-1}y,y)}(d(y,py)+C)\right)
\]
for a constant $C>0$ independent of $\epsilon'$ and $k$.

The neighborhood $V(\epsilon')$ is smaller and smaller as $\epsilon'\rightarrow 0$, as if $\eta^-\eta^+$ intersects $\mathcal{H}_k(\epsilon')$ then it has to intersect $\mathcal{H}_k(\epsilon)$. Then for the $p$ summands considered for $\epsilon'$ we have

\[d(y,py)\leq d(x,px)+C',\quad \mu^k_y(\mathcal{D}) \leq C'e^{-\delta(\Ga_k)d(x,y)}\mu^k_x(\mathcal{D})
\]
for $C'$ constant independent of $\epsilon'$ and $k$. We always have the bound $d(y,py)\geq d(x,px) - 2d(x,y)$ by triangular inequality and the fact that $p$ is an isometry.

Putting altogether, we have that
\begin{equation}
\label{tail}
m^k_{\rm BM}(T^1(\mathcal{H}_k(\epsilon')/\mathcal{P}_k)) \leq  C''(\mu^k_x(\mathcal{D}))^2\left(\sum_{p\in\mathcal{P}_k, p^{-1}x\in V(\epsilon')} e^{-\delta(\Ga_k)d(p^{-1}x,x)}(d(x,px)+C'')\right)
\end{equation}
for a constant $C''>0$ independent of $\epsilon'$ and $k$. Recall that  when $\delta(\Ga_k)$ is strictly bigger than $(n-1)/2$, by \cite[Theorem 6.1]{McMullen99}, the tails of the series $\sum_{p\in\mathcal{P}_k}e^{-\delta(\Ga_k)d(p^{-1}x,x)}$ are uniformly small. Specifically, for any $\eta>0$ there exists a neighborhood $U\subset  \HH^{n}$ of $\xi$ so that

\[\sum_{p\in\mathcal{P}_k, \,px\subset U}e^{-\delta(\Ga_k)d(p^{-1}x,x)} < \eta,
\]
for $k$ sufficiently large. We also have that the tails of the series $\sum_{p\in\mathcal{P}_k}e^{-\delta(\Ga_k)d(p^{-1}x,x)}(d(x,px)+C'')$ are uniformly small, as  $d(x,px)$ is uniformly dominated by $e^{cd(p^{-1}x,x)}$ for any $c>0$. Hence by taking $\epsilon'$ sufficiently small, the right hand side of \eqref{tail} corresponds to a smaller tail of the series $\sum_{p\in\mathcal{P}_k}e^{-\delta(\Ga_k)d(p^{-1}x,x)}(d(x,px)+C'')$. Thus, by applying \cite[Theorem 6.2]{McMullen99} for the sequence of exponents $\delta(\Gamma_k)-c$, the right hand side of  will be arbitrarily small for $\epsilon'$ sufficiently small and $k$ sufficiently large.

\end{proof}

Next we use Proposition \ref{lem:ControlThinPart} to prove the convergence of the Bowen-Margulis measures. The following proposition  is a restatement of Proposition \ref{prop:margulismeasure}.

\begin{proposition}
\label{prop:convergenceBowen}
Suppose that  $(M_k=\HH^{n}/ \Ga_{k})_{k\in\mathbb{N}}$ is a sequence of hyperbolic manifolds which is strongly convergent to a geometrically finite hyperbolic manifold $M=\HH^{n}/ \Ga$ with $\delta(\Ga)>(n-1)/2$. Let $m^k_{\rm BM}, m_{\rm BM}$ be the Bowen-Margulis measures on $T^{1}M_k$ and $T^{1}M$, respectively. For $r>0$ we denote by $M_{k}^{<r}\subset M_k, M^{<r}\subset M$ the set of points with injectivity radius less than r. Then for any $r>0$ we have
\[\lim_{k\rightarrow\infty}\int_{T^1M_k^{<r}}dm^k_{\rm BM} \rightarrow \int_{T^1M^{<r}}  dm_{\rm BM}.
\]
Moreover, by taking $r$ sufficiently large we have that
\[\Vert m^k_{\rm BM}\Vert \rightarrow \Vert m_{\rm BM}\Vert.\]
\end{proposition}
\begin{proof} Denote $M^{a,b}=M^{>a}\cap M^{<b}$. Take $U_1,\ldots U_m\subset M$ balls with compact closure, whose union covers $C(M)^{\epsilon,r} = C(M) \cap M^{\epsilon,r}$. 
Take $\bar{\varphi}_1,\ldots,\bar{\varphi}_m$ partition of unity subordinated to $U_1,\ldots U_m$, in the sense that $\bar{\varphi}=\sum_{i=1}^m \bar{\varphi}_i$ has support contained in $M^{\epsilon-\eta,r+\eta}$ and is identically equal to $1$ in $C(M)^{\epsilon,r}$, for some arbitrarily small $\eta>0$. Let $\tilde{U_i}$ be a lift of $U_i$ in $\mathbb{H}^n$ such that the union covers a fundamental domain of $M$. We denote $\varphi_i$ a compactly supported function subordinated to $\tilde{U}_i$ such that $\varphi_i=\bar{\varphi}_{i}\circ \textup{Proj}$.

Then since the Patterson-Sullivan measures $\mu^k_{x_0}$ converge weakly to $\mu_{x_0}$, the critical exponents $\delta_k=\delta(\Ga_k)$ converge to $\delta=\delta(\Ga)$, and we can express the Bowen-Margulis measures as $d\tilde{m}^k_{\rm{BM}}(v)=e^{-\delta_k(\beta_{v_{-}}(\pi(v), x_{0})+\beta_{v^{+}}(\pi(v), x_{0}))}d\mu^k_{x_{0}}(v_{-})d\mu^k_{x_{0}}(v_{+})dt$, then for $k$ sufficiently large we have

\begin{equation}\label{eq:1stBMS}
\bigg|\int_{T^1\tilde{U_i}}\varphi_id\tilde{m}^k_{\rm BM} - \int_{T^1\tilde{U_i}}\varphi_i d\tilde{m}_{\rm BM}\bigg| <\alpha,
\end{equation}
for some small $\alpha>0$.

By Proposition \ref{lem:ControlThinPart} we have that $\int_{T^1 M^{<\epsilon}_k} dm^k_{\rm BM}, \int_{T^1 M^{<\epsilon}} dm_{\rm BM}<\alpha$, and by construction we have that

\begin{equation}\label{eq:2ndBMS}
\bigg|\int_{T^1M^{<r}}dm_{\rm BM} - \sum_{i=1}^m\int_{T^1\tilde{U_i}}\varphi_id\tilde{m}_{\rm BM}\bigg| < \int_{T^1M^{<\epsilon}} dm_{\rm BM} + \int_{T^1M^{r,r+\eta}} dm_{\rm BM}.
\end{equation}

Now, since $M_k$ converges strongly to $M$, for $k$ large we have that $\sum_{i=1}^m\int_{T^1\tilde{U_i}}\varphi_id\tilde{m}^{k}_{\rm BM}$ is bounded between $(1-\alpha)\int_{T^1M_k^{\epsilon+\eta, r-\eta}}dm^k_{\rm BM}$ and $(1+\alpha)\int_{T^1M_k^{\epsilon-\eta, r+\eta}}dm^k_{\rm BM}$. Hence

\begin{equation}\label{eq:3rdBMS}
\bigg|\int_{T^1M_k^{<r}}dm^k_{\rm BM} - \sum_{i=1}^m\int_{T^1\tilde{U_i}}\varphi_id\tilde{m}^k_{\rm BM}\bigg| < \int_{T^1M_k^{<\epsilon}} dm^k_{\rm BM} + \alpha\int_{T^1M_k^{\epsilon,r}} dm^k_{\rm BM}  + \int_{T^1M_k^{r,r+\eta}} dm^k_{\rm BM}.
\end{equation}

By a similar partition of unity argument, we can show that for any $0<a<b$ there exists $\eta_0>0$ sufficiently small so that for any $k\gg1$ sufficiently large we have that

\begin{equation}\label{eq:4thBMS}
\int_{T^1M_k^{a,b}}dm^k_{\rm BM} < 2 \int_{T^1M^{a-\eta_0,b+\eta_0}}dm_{\rm BM}.
\end{equation}

Finally, we have to see that the function $(a,b)\mapsto \int_{T^1M^{a,b}}dm^k_{\rm BM}$ is continuous. Because of monotonicity this reduces to prove that for any $r>0$,  $\int_{T^1\partial M^{<r}}dm^k_{\rm BM}=0$. Indeed, the lift $\partial \tilde{M}^{<r}\subseteq\mathbb{H}^n$ is contained the union of tubes around closed geodesics  of length $\leq r$ (considering parabolic cusps corresponding to  $0$ length geodesics). For core geodesics of length strictly less than $r$, these tubes are strictly convex and hence the boundaries intersect any geodesic in a discrete set. If we happen to have a geodesic of length $r$, then the intersection of $\partial \tilde{M}^{<r}$ with any geodesic is a discrete set, unless the geodesic is equal to the geodesic axis. In either case, the set $\partial \tilde{M}^{<r}\subseteq\mathbb{H}^n$ has zero measure for the Bowen-Margulis measure $d\tilde{m}_{\rm{BM}}(v)= e^{-2\delta(v_-|v_+)_{x_0}}d\mu_{x_{0}}(v_{-})d\mu_{x_{0}}(v_{+})dt$, as for almost every geodesic line $\ell$ the intersection $\partial \tilde{M}^{<r}\cap\ell$ has length $0$.

Applying the  triangular inequality, replacing equations (\ref{eq:1stBMS}), (\ref{eq:2ndBMS}), (\ref{eq:3rdBMS}), and then using Proposition \ref{lem:ControlThinPart},  (\ref{eq:4thBMS})  (for sufficiently large $k$ and $\eta$ sufficiently small) we have that

\begin{equation}
\begin{aligned}
    \bigg| \int_{T^1M_k^{<r}}dm^k_{\rm BM} - \int_{T^1M^{<r}}  dm_{\rm BM} \bigg| &< \bigg|\int_{T^1 M_k^{<r}}dm^k_{\rm BM} - \sum_{i=1}^m\int_{T^1\tilde{U_i}}\varphi_id\tilde{m}^k_{\rm BM}\bigg| \\&+ \sum_{i=1}^m \bigg|\int_{T^1\tilde{U_i}}\varphi_{i}d\tilde{m}^k_{\rm BM} - \int_{T^1\tilde{U_i}}\varphi_i d\tilde{m}_{\rm BM}\bigg| \\&+ \bigg|\int_{T^1M^{<r}}dm_{\rm BM} - \sum_{i=1}^m\int_{T^1\tilde{U_i}}\varphi_id\tilde{m}_{\rm BM}\bigg|\\
    &<\int_{T^1M_k^{<\epsilon}} dm^k_{\rm BM} + \alpha\int_{T^1M_k^{\epsilon,r}} dm^k_{\rm BM}  + \int_{T^1M_k^{r,r+\eta}} dm^k_{\rm BM} \\&+ m\alpha + \int_{T^1M^{<\epsilon}} dm_{\rm BM} + \int_{T^1M^{r,r+\eta}} dm_{\rm BM}\\
    &<(m+2)\alpha + 2\alpha\int_{T^1M^{<r+\eta}} dm_{\rm BM} + 3 \int_{T^1M^{r-\eta,r+2\eta}} dm_{\rm BM}
\end{aligned}
\end{equation}
which goes to $0$ as $k\rightarrow+\infty$, and $\alpha,\eta\rightarrow0$.

\end{proof}

The last part of the section is to prove the convergence of  skinning measures.

\medskip
\noindent
{\bf Proof of Corollary \ref{coro:convergenceskinning}}:
Observe that since we have strong convergence for well-positioned convex sets $D_k\rightarrow D$, we can take lifts $\widetilde{D_k}, \tilde{D}\subset \mathbb{H}^n$ and compact sets $E_k\subset \partial\widetilde{D_k}, E \subset \partial\tilde{D}$ so that $E_k$, $E$ are fundamental domains for the support of $\sigma^{\pm}_{\partial D_k}, \sigma^{\pm}_{\partial D}$ (respectively) and $E_k$ converges strongly to $E$. We can further assume there exists a set $F\subset \mathbb{S}^{n-1}$ so that $P_{D_k}(F)$, $P_D(F)$ cover $E_k$ and $E$ (respectively) on its interior, see \eqref{definemap} for the definition of the maps $P_{D_{k}}$ and $P_{D}$. Under this assumptions we have

\[\Vert\sigma^{\pm}_{\partial D_k}\Vert \leq \tilde{\sigma}^{\pm}_{\widetilde{D_k}}(P_{D_k}(F)) \rightarrow \tilde{\sigma}^{\pm}_{\tilde{D}}(P_D(F)).
\]

Reducing the set $F$ so that $\tilde{\sigma}^{\pm}_{\widetilde{D_k}}(P_{D_k}(F)\setminus E_k), \tilde{\sigma}^{\pm}_{\tilde{D}}(P_D(F)\setminus E)$ are arbitrarily small, we then have

\[\Vert\sigma^{\pm}_{\partial D_k}\Vert=\tilde{\sigma}^{\pm}_{\widetilde{D_k}}(E_k) \rightarrow \tilde{\sigma}^{\pm}_{\tilde{D}}(E) =\Vert\sigma^{\pm}_{\partial D}\Vert,
\]
which proves the first statement

The relative result for subsets $\Omega_k, \Omega$ is proved by taking  the fundamental domains $E'_k\subset \partial\widetilde{D_k}$, $E'\subset \partial\tilde{D}$ for the support of $\sigma^{\pm}_{\partial \Omega_{k}}, \sigma^{\pm}_{\partial \Omega}$ (respectively) and arguing as above.

\qed

\section{Application: uniform orthogeodesic counting}\label{uniformcounting}
In this section, we use the results of uniform spectral gap and convergence of the Bowen-Margulis and skinning measures in Section  \ref{sec:eigenvalues} and Section \ref{sec:measure} to prove Theorem \ref{thm:uniformcounting}. Suppose that $D^{+}, D^{-}$ are well-positioned convex subsets of a hyperbolic manifold $M=\HH^{n}/ \Ga$. A \emph{common perpendicular} from $D^{-}$ to $D^{+}$ is a locally geodesic path in $M$ which starts perpendicularly from $D^{-}$ and arrives perpendicularly to $D^{+}$.  For any $t\geq 0$, let $\N_{D^{-}, D^{+}}(t)$ be the cardinality of the set of common perpendiculars from  $D^-$ to $D^+$ with length at most $t$.

As before, $(M_k=\HH^{n}/ \Ga_k)_{k\in\mathbb{N}}$  is a sequence of hyperbolic manifolds which converges strongly to a geometrically finite manifold $M=\HH^{n}/ \Ga$, so that we have well-positioned convex subsets $D^\pm_k\subset M_k$ that strongly converge to $D^\pm\subset M$. Before the proof, we need to introduce the following notations.

Given $v\in T^{1}\HH^{n}$, the \emph{strong stable/unstable} manifold is defined as 
$$W^{\pm}(v)=\{v'\in T^{1}\HH^{n}: d(v(t), v'(t))\rightarrow 0 \text{ as } t\rightarrow \pm \infty\},$$
which is equipped with Hamenst\"adt's distance function $d_{W^{\pm}(v)}$, see \cite{Ursula89, ParkkonenPaulin21}. Then given any constant $r>0$, for all $v\in T^{1}\HH^{n}$, we can define the open ball of radius $r$ centered at $v$ in the strong stable/unstable manifold in the following
$$B^{\pm}(v, r)=\{v'\in W^{\pm}(v): d_{W^{\pm}(v)}(v, v')<r\}.$$
Given any $v\in T^{1}\HH^{n}$, and $\eta, \eta'>0$, let
$$V^{\pm}_{v, \eta,\eta'}=\bigcup_{s\in \left[-\eta, \eta\right]} g^{s} B^{\pm}(v, \eta').$$
Given a proper closed convex subset $D$ of $\HH^{n}$, for all subsets $\Omega^{-}$ of $\partial^{1}_{+}D$ and $\Omega^{+}$ of $\partial^{1}_{-}D$, let 
$$\V_{\eta, \eta'}(\Omega^{\pm})=\bigcup_{v\in \Omega^{\pm}} V^{\mp}_{v, \eta, \eta'}. $$

By using the projection map $\pi:T^{1}\HH^{n}\rightarrow \HH^{n}$, the strong stable/unstable manifold $W^{\pm}(v)$ projects to the \emph{stable/unstable horosphere} of $v$ centered at $v_+$ and $v_-$, denoted by $H_{\pm}(v)=\pi(W^{\pm}(v))$. The corresponding \emph{horoball} bounded by $H_{\pm}(v)$ is denoted by $HB_{\pm}(v)$. Following the notation in \cite{ParkkonenPaulin21}, we let 
$$\mu_{W^{+}(v)}=\tilde{\sigma}^{-}_{HB_{+}(v)} \quad \text{ and }  \quad \mu_{W^{-}(v)}=\tilde{\sigma}^{+}_{HB_{-}(v)}$$
denote the skinning measures on the strong stable/unstable manifolds $W^{\pm}(v)$.

\begin{definition}
Given a discrete isometry subgroup $\Ga<\Isom(\HH^{n})$, we say $(\HH^{n}, \Ga)$ has \emph{radius-continuous strong stable/unstable ball masses} if, for every $\epsilon>0$, and $r\geq 1$ close enough to $1$, 
$$\mu_{W^{\pm}(v)}(B^{\pm}(v, r))\leq e^{\epsilon} \mu_{W^{\pm}(v)} (B^{\pm}(v, 1)),$$
for all  $v\in T^{1}\HH^{n}$ where $B^{\pm}(v, 1)$ meets the support of $\mu_{W^{\pm}(v)}$.

\end{definition}

The following proposition proves that the radius-continuous property of the strong stable/unstable ball masses can be taken uniformly along a strongly convergent sequence of geometrically finite hyperbolic manifolds. 

\begin{proposition}\label{prop:radiuscontinuity}
Suppose  $(M_{k}=\HH^{n}/ \Ga_{k})_{k\in\mathbb{N}}$ is a sequence of hyperbolic  manifolds which strongly converges to a geometrically finite hyperbolic manifold $M=\HH^{n}/ \Ga$ with $\delta(\Ga)>(n-1)/2$. Let $(D_{k}\subseteq M_k)_{k\in\mathbb{N}}$ be a sequence of well-positioned convex subsets in $M_{k}$ which strongly converges to a well-positioned convex subset $D$, with lifts to $\mathbb{H}^n$ denoted by $\widetilde{D_k}, \widetilde{D}$, respectively. Let $\Omega_k^{\mp}\subseteq \partial^1_\pm \widetilde{D_k}$, $\Omega^{\mp}\subseteq \partial^1_\pm \widetilde{D}$ be compact sets so that $\Omega_k^{\mp}$ converges strongly to $\Omega^{\mp}$. Then there exists sufficiently large $R>0$ so that for any $\epsilon$ we have $\eta=\eta(\epsilon,R)>0$ satisfying that

\[\mu^k_{W^\pm(v)}(B^{\pm}(v,(1+r)R)) \leq e^\epsilon \mu^k_{W^\pm(v)}(B^{\pm}(v,R))
\]
for any $v\in\Omega_k^{\mp}$, $0<r<\eta$.

\end{proposition}

\begin{proof}
Let's prove that case for $\Omega^{+}$, and the proof for $\Omega^-$ is similar. 
As done in the proof of \cite[Proposition 6.2]{Roblin03} (using \cite[Section 3.1]{Roblin00}), the function $(v,R)\mapsto \mu_{W^\pm(v)}(B^\pm(v,R))$ is continuous for $v\in T^1\mathbb{H}^n, R>0$, as well as $\Gamma$-invariant. Moreover, since $\Omega^+$ is compact, there exists $R>0$ sufficiently large so that the function $v\mapsto \mu_{W^-(v)}(B^-(v,R))$ is a uniformly continuous positive function in some neighborhood of $\Omega^-$. It suffices then to prove the statement for sufficiently large $k$.

Denote by $A(v,R,r) = B^-(v,(1+r)R)\setminus B^-(v,R) \subset W^-(v)$ the annulus in $W^-(v)$ with center $v$ between radius $R, (1+r)R$.  We will show that there exists $m>0$ and function $\eta(\epsilon)>0$ so that for $k$ large, $v\in \mathcal{V}_{\eta,\eta} (\Omega^{+})$ and $0<r<\eta$ the following two statements hold

\begin{enumerate}
    \item\label{item:ballmass} $ \mu^k_{W^-(v)}(B^-(v,R))\geq m$,
    
    \item\label{item:annulimass} $ \mu^k_{W^-(v)}(A(v,R,r)) < \epsilon$.
\end{enumerate}
Then it is clear that the statement follows from (\ref{item:ballmass}) and (\ref{item:annulimass}) by making $\epsilon$ arbitrarily small. Now we prove items (\ref{item:ballmass}) and (\ref{item:annulimass}) respectively. 

\begin{enumerate}
    \item For a vector $u\in T^{1}\HH^{n}$, we define a function $P_{u}: W^{-}(u)\rightarrow \geo\HH^{n}$ where $P_{u}(v)$ is the endpoint of the bi-infinite geodesic $u_{-}\pi(v)$  different from $u_{-}$ as shown in Figure \ref{F1}. Since $\mathcal{V}_{\eta,\eta} (\Omega^+)$ has compact closure, we can take finitely many $v_i \in \mathcal{V}_{\eta,\eta} (\Omega^+)$ so that for any $u\in\mathcal{V}_{\eta,\eta} (\Omega^+)$,  there exists $v_i$ such  that
    \[ P_u^{-1}P_{v_i}(B^-(v_i,R/2))\subseteq B^-(u,R).
    \]
    Moreover, we can assume that the conformal factor between $\mu^k_{W^-(v_i)}$ and $\mu^k_{W^-(u)}$ at the sets $B^-(v_i,R/2)$, $P_u^{-1}P_{v_i}(B^-(v_i,R/2))$ is between $1/2$ and $2$.
    This can be done uniformly for all $k$ by following \cite[Subsection 1.H]{Roblin03}. By taking $\eta$ small we can assume that $\mu_{W^-(v_i)}(B^-(v_i,R/2)) > 2m$ for some fixed $m>0$ and for any $v_i\in \mathcal{V}_{\eta,\eta} (\Omega^+)$. Then by weak-convergence of measures, we have that for any $v_i$ (and large $k$)  $\mu^k_{W^-(v_i)}(B^-(v_i,R/2)) > 2m$. Then it follows that
    \[\mu^k_{W^-(u)}(B^-(u,R)) \geq \mu^k_{W^-(u)}(P_u^{-1}P_{v_i}(B^-(v_i,R/2))) \geq \frac12 \mu^k_{W^-(v_i)}(B^-(v_i,R/2)) > m.
    \]
    
    \item Since $\mathcal{V}_{\eta,\eta} (\Omega^+)$ has compact closure and $\delta(\Ga)>(n-1)/2$, given $\epsilon>0$ we can take $\eta$ small enough so that for $v\in\mathcal{V}_{\eta,\eta} (\Omega^+)$ we have that $\mu_{W^-(v)}(A(v,R,5\eta))<\epsilon$. We will take again a finite collection of vectors $v_i$, although now they need to satisfy the following list of properties. 
    \begin{itemize}
        \item The finite collection of $v_i$ is taken so that $B^-(v_i,4\eta)\subset A(v,R,5\eta)$ for some $v\in \mathcal{V}_{\eta,\eta} (\Omega^+)$. Denote their total number by $C_2$,
        \item For any $v\in \mathcal{V}_{\eta,\eta} (\Omega^+)$ and any $B^-(u,2\eta)\subset A(v,R,5\eta)$ we have that
        \[ P_u^{-1}P_{v_i}(B^-(v_i,4\eta))\supseteq B^-(u,2\eta)
        \]
        with conformal factor bounded between $\frac12$ and $2$.
    \end{itemize}
    Take sufficiently large $k$ so that $\mu^k_{W^-(v_i)}(B^-(v_i,4\eta)) \leq  \mu_{W^-(v_i)}(B^-(v_i,4\eta)) + \zeta$ for $\zeta$ small still to be determined.
    
    Let $v\in \mathcal{V}_{\eta,\eta} (\Omega^+)$. Cover $A(v,R,\eta)$ by finitely many disjoint measurable sets $B_j$, so that each $B_j$ is contained in a ball $B^-(u_j,2\eta)$ inside of $A(v,R,5\eta)$. Then by the second bullet point, for each $u_j$ we choose $v_i$ so that
    \[ P_{u_j}^{-1}P_{v_i}(B^-(v_i,4\eta))\supseteq B^-({u_j},2\eta).
    \]
    Observe that each $v_i$ can only be repeatedly selected less than $C_3$ times, for some constant $C_3$ depending only on the dimension $n$. Then we have the following chain of inequalities, which follow from the covering $\{B_j\}$ of $A(v,R,r)$, the inclusion $P_{u_j}^{-1}P_{v_i}(B^-(v_i,4\eta))\supseteq B^-({u_j},2\eta) \supseteq B_j$, the bound on the conformal factor of $P_u^{-1}P_{v_i}$, the convergence $\mu^{k}_{W^-(v)} \rightarrow \mu_{W^-(v)}$, the inclusion $B^-(v_i,4\eta)\subset A(v,R,5\eta)$, and the bound on the cardinality of the finite set of $v_i$'s.
    \begin{equation}
    \begin{aligned}
    \mu^k_{W^-(v)}(A(v,R,r)) &\leq \sum_j \mu^k_{W^-(v)}(B_j) \leq C_3\sum_i \mu^k_{W^-(v)}(P_{u_j}^{-1}P_{v_i}(B^-(v_i,4\eta)))\\
    &\leq 2 C_3\sum_i \mu^k_{W^-(v)}(B^-(v_i,4\eta)) \leq 2C_3 \sum_i \left(\mu_{W^-(v)}(B^-(v_i,4\eta)) + \zeta\right)\\
    &\leq 4C_3\sum_i \left(\mu_{W^-(v)}(A(v,R,5\eta)) + \zeta\right) \leq 4C_2 C_3 (\epsilon + \zeta)
    \end{aligned}
    \end{equation}
    which is arbitrarily small for $\eta$ small and $k$ large.
    
\end{enumerate}
\end{proof}

\begin{figure}
\centering
\includegraphics[width=2.0in]{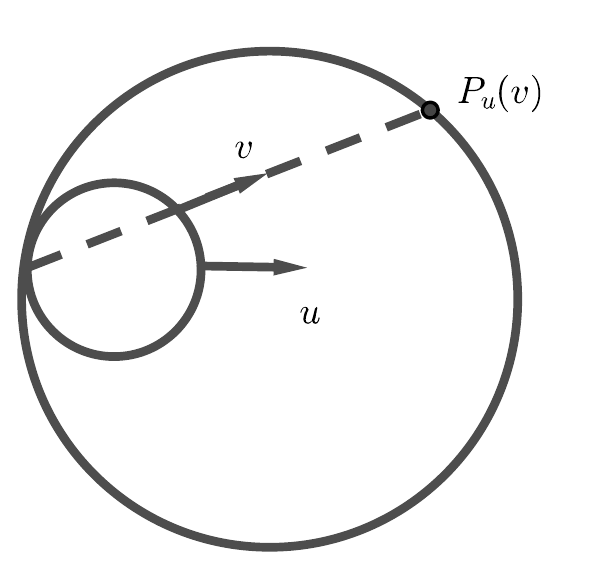}
\caption{\label{F1}}
\end{figure}

Now we state and sketch the general uniform orthogeodesic counting for convergent sequences of convex sets in strongly convergent hyperbolic $n$-manifolds. For a thorough presentation, we refer the reader to Theorem \ref{thm:appendix} in the Appendix. 
\noindent
\begin{theorem}
\label{mainthm:counting}
Suppose that $(M_{k}=\HH^{n}/ \Ga_{k})_{k\in\mathbb{N}}$ is a sequence of hyperbolic manifolds  which strongly converges to a geometrically finite hyperbolic manifold $M=\HH^{n}/\Ga$ with $\delta(\Ga)>(n-1)/2$. Let $(D^{\pm}_{k})_{k\in\mathbb{N}}$ be a sequence of well-positioned convex subsets in $M_k$ which converges strongly to a well-positioned convex subset $D^{\pm}$ in $M$, respectively. Then we can count $\N_{D^{-}_k, D^{+}_k}(t)$ uniformly, in the sense that 
$$\N_{D^{-}_k, D^{+}_k}(t)\approx\dfrac{||\sigma^{+}_{D^{-}_k}||\cdot || \sigma^{-}_{D^{+}_k}||}{\delta(\Gamma_k) ||m^k_{\rm{BM}}||} e^{\delta(\Ga_k) t}$$
up to a multiplicative error uniformly close to 1 along the sequence as $t$ gets larger, and with $||\sigma^{\mp}_{D^{\pm}_k}||, ||m^k_{\rm{BM}}||,\delta(\Gamma_k) $ converging to $||\sigma^{\pm}_{D^{\mp}}||, ||m_{\rm{BM}}||,\delta(\Gamma)$, respectively. In particular, for $n=3$, there exist constants $A>0,0<b<2$ so that
$$\N_{D^{-}_k, D^{+}_k}(t)\leq Ae^{bt}$$
\end{theorem}

\begin{proof}
There is an explicit counting formula of $\N_{D^{-}, D^{+}}(t)$ for orthogeodesic arcs between two convex sets $D^{\pm}$ given  in \cite[Theorem 3]{ParkkonenPaulin21}:
$$\N_{D^{-}, D^{+}}(t)=\dfrac{||\sigma^{+}_{D^{-}}||\cdot || \sigma^{-}_{D^{+}}||}{\delta ||m_{\rm{BM}}||} e^{\delta(\Ga) t}(1+O(e^{-\kappa t})).$$
This formula holds under the assumption that $(\HH^{n}, \Ga)$ has radius-continuous strong stable/unstable masses.  The constant $O(\cdot)$ and  the parameter $\kappa$ depends on $\Ga$, the convex sets  $D^{\pm}$, the speed of mixing, and the property of  radius-continuous strong stable/unstable masses.

By Proposition \ref{prop:convergenceBowen} and Corollary \ref{coro:convergenceskinning}, the Bowen-Margulis measure and the skinning measures converge to the ones of the limit manifold $M$ weakly. The critical exponent $\delta(\Ga_{k})$ converges to $\delta(\Ga)$ \cite[Theorem 1.5]{McMullen99}. The convergence of the speed of mixing is controlled by the spectral gap \cite{EdwardsOh21}. Hence this quantity also converges to the one of the limit manifold by Theorem \ref{thm:uniformgap}. Therefore, it suffices to prove the sequence $\Ga_k$ and the limit $\Ga$ have uniform radius-continuous strong stable/unstable ball masses property, which follows from Proposition \ref{prop:radiuscontinuity}.

\end{proof}

\begin{remark}
Careful readers might notice that \cite[Theorem 3]{ParkkonenPaulin21} has the assumption that the manifold has \emph{radius-H\"older-continuous strong stable/unstable ball masses}, which is not satisfied by the strongly convergent sequence of hyperbolic manifolds $M_k$ and the limit manifold $M$. However, this assumption can be replaced by  the property of  radius-continuous strong stable/unstable masses \cite[Lemma 11]{ParkkonenPaulin21}, and the uniform radius-continuity suffices to control the error term in our setting.  We write down the details about the replacement in the Appendix for readers' convenience, and most of the arguments follow from \cite{ParkkonenPaulin21}.
\end{remark}

\medskip
\noindent
{\bf Proof of Theorem  \ref{thm:uniformcounting}}: By Example \ref{ex:convergentset}, connected components $D_{k}^{\pm}$ in the thin part of $M_{k}$ are well-positioned convex sets that are strongly  convergent to the well-positioned convex sets $D^{\pm}$ (respectively). By Theorem \ref{mainthm:counting}, there is a uniform counting formula for orthogeodesics between $D^{-}_{k}$ to $D^{+}_{k}$ along the sequence. This proves item (1). Similarly, for small $r>0$, the radius $r$ embedded balls centered at $x_{k}$ are well-positioned convex subsets which are strongly convergent to the embedded $r$-ball centered at $x$. In that case, let $D^{+}_{k}=D^{-}_{k}$ be the radius $r$ ball at $x_k$, and $D^+=D^-$ be the radius $r$ ball at $x$. Observe that if we change the radius $r>0$ to a radius $s>0, s<r$ we have a one-to-one correspondence between the set of orthogeodesics by extending/shortening the geodesic arcs. Such correspondence takes an orthogeodesic of length $\ell$ to its extension of length $\ell+2(r-s)$. Hence applying Theorem \ref{mainthm:counting} again and  making $s$ arbitrarily small (or equivalently, translating by $2r$ the counting function for the balls of radius $r$), we obtain the uniform counting for geodesic loops based at $x_k$ along the sequence. 

\qed

\medskip
\noindent
{\bf Proof of Corollary \ref{coro:simple}}: As explained for instance by Roblin in \cite[Chapter 5]{Roblin03}, one can deduce an asymptotic counting of closed primitive geodesics in manifolds with  negative pinched curvature from the asymptotic counting of orbit distance (i.e., geodesic loops), which only depends on the geometry of the universal cover. Namely, if $\mathcal{G}_{M}(\ell)$ is the set of closed primitive geodesics in $M$ of length less than $\ell>0$, then \cite[Corollary 5.3]{Roblin03}
\[\#\mathcal{G}_{M}(\ell) \approx \frac{e^{\delta\ell}}{\delta\ell} \text{ as } \ell\rightarrow+\infty.
\]

Combining with the uniform counting of geodesic loops (Theorem \ref{thm:uniformcounting}), we obtain the uniform counting of closed primitive geodesics along a strongly convergent sequence of hyperbolic manifolds. 

\qed

\section*{Appendix}

Let's start with notations needed in the Appendix. Recall that $P_{D}: \HH^{n}\cup (\geo \HH^{n}\setminus \geo D)\rightarrow D$ is the closest point map defined in Section \ref{subsec:BMmeasure} for any nonempty proper closed convex subset $D$ in $\HH^{n}$. Let $P^{+}_{D}$ denote the inverse of the restriction to $\partial^{1}_{+}D $ of the \emph{positive endpoint map} $v\mapsto v_{+}$, which is a homeomorphism from $\geo \HH^{n}\setminus \geo D$ to $\partial^{1}_{+} D$. It is a natural lift of $P_{D}$ such that $\pi\circ P^{+}_{D}=P_{D}$ on $\geo \HH^{n}\setminus \geo D$ where $\pi: T^{1}\HH^{n}\rightarrow \HH^{n}$. Similarly, one can define $P^{-}_{D}=\iota \circ  P^{+}_{D}$, where $\iota: T^{1}\HH^{n}\rightarrow T^{1} \HH^{n}$ is the \emph{antipodal flip map} given by $\iota v=-v$. 

Define 
$$\U_{D}^{\pm}=\{v\in T^{1}\HH^{n}: v_{\pm}\notin \geo D\}.$$
This is an open set in $T^{1}\HH^{n}$ which is invariant under the geodesic flow and satisfies the $\U^{\pm}_{\gamma D}=\gamma \U^{\pm}_{D}$ for any $\ga\in \Isom(\HH^{n})$. Define a fibration $f_{D}^{+}: \U^{+}_{D}\rightarrow \partial^{1}_{+} D$ as the composition of the positive endpoint map and $P^{+}_{D}$. Given $w\in \partial^{1}_{+} D$, the fiber of $w$ for $f^{+}_{D}$ is the set 
$$W^{0+}(w)=\{v\in T^{1}\HH^{n}: v_{+}=w_{+}\}.$$
Similarly, one can define a fibration $f^{-}_{D}=\iota\circ f^{+}_{D}\circ \iota: \U^{-}_{D}\rightarrow \partial^{1}_{-} D$ and the fiber $W^{0-}(w)=\{v\in T^{1}\HH^{n}: v_{-}=w_{-}\}$.

Suppose that $D^{\pm}$ are two well-positioned convex subsets in $M=\HH^{n}/ \Ga$, and $\psi^{\pm}\in C^{\infty}_{0}(T^{1}M)$ are compactly supported functions. Let 
$$\mathcal{N}_{\psi^{-}, \psi^{+}}(t)=\sum_{\lambda, 0<\ell_\lambda\leq t} \psi^{-}(v^{-}_{\lambda})\psi^{+}(v^{+}_{\lambda})$$
where the sum is taken over all common perpendiculars $\lambda$ between $D^{-}$ and $D^{+}$ whose initial vector $v^{-}_{\lambda}$ belongs to $\partial^{1}_{+} D^{-}$ and the terminal vector $v^{+}_{\lambda}$ belongs to $\partial^{1}_{-} D^+$, and the length $\ell_{\lambda}\leq t$. 

In order to count orthogeodesics between $D^-$ and $D^+$, we can parametrize the set of orthogeodesics by a quotient of $\Gamma$ up to a choice of basepoint. Denote by $\widetilde{D^\pm}$ the lifts of $D^\pm$ in $\mathbb{H}^n$, and distinguish two components $D^\pm_0\subset \widetilde{D^\pm}$. Then for each $\gamma\in\Gamma$ we can consider the projection to $M$ of the unique orthogeodesic between $D^-_0$ and $\gamma D^+_0$ such that the closures of $D^{-}_{0}$ and $\gamma D^{+}_{0}$ in $\HH^{n}\cup \geo \HH^{n}$ have empty intersection. It is a simple exercise to see that $\gamma_1,\gamma_2\in \Gamma$ map to the same orthogeodesic if and only if there exists $g^\pm\in Stab(D^\pm_0)$ so that $\gamma_1 = g^-\gamma_2 g^+$. Hence we can parametrize orthogeodesic by taking the quotient $\Gamma/\sim:= \Gamma/\lbrace \gamma_1 = g^-\gamma_2 g^+, g^\pm\in Stab(D^\pm_0) \rbrace$. Although this labeling depends on the choice of $D^\pm_0$, we will always work once this decision has been made. We use  $v^\mp_{\gamma}\in \partial^1_\pm D^\mp$ to denote the unit tangent vector of $\gamma$ at the start/end.

\begin{theorem}\label{thm:appendix}
Suppose that $(M_{k}=\HH^{n}/ \Ga_{k})_{k\in\mathbb{N}}$ is a sequence of hyperbolic manifolds  which strongly converges to a geometrically finite hyperbolic manifold $M=\HH^{n}/\Ga$ with $\delta(\Ga)>(n-1)/2$. Let $(D^{\pm}_{k})_{k\in\mathbb{N}}$ be a sequence of well-positioned convex subsets in $M_k$ which strongly converges to $D^{\pm}$ in $M$, respectively. Let as well $(\psi^\pm_k \in C^\infty_0(T^1M_k))_{k\in\mathbb{N}}$, $\psi^\pm\in C^\infty_0(T^1M)$ be compactly supported functions so that $\psi^\pm_k$ converges strongly to $\psi^\pm$, respectively. Then for any $\epsilon>0$ there exists $t_0 = t_0(\epsilon), k_0 = k_0(\epsilon)>0$  so that for any $t>t_0$, $k>k_0$ we have that

\begin{equation}\label{eq:assymptoticcounting}
    \frac{\sigma^{+}_{D^{-}_k}(\psi^-_k)\cdot  \sigma^{-}_{D^{+}_k}(\psi^+_k)}{\delta(\Gamma_k) ||m^k_{\rm{BM}}||} -\epsilon \leq \frac{N_{\psi^{-}_k, \psi^{+}_k}(t)}{e^{\delta(\Ga_k) t}} \leq \frac{\sigma^{+}_{D^{-}_k}(\psi^-_k)\cdot  \sigma^{-}_{D^{+}_k}(\psi^+_k)}{\delta(\Gamma_k) ||m^k_{\rm{BM}}||} + \epsilon.
\end{equation}
Here $\sigma^{+}_{D^{-}_k}(\psi^-_k)=\int_{\partial^{1}_{+} D_{k}^{-}}\psi^-_k d\sigma^+_{k}$, and $\sigma^{-}_{D^{+}_k}(\psi^+_k)$ is similarly defined.

\end{theorem}

\begin{proof}
Since both terms in (\ref{eq:assymptoticcounting}) are bilinear in $\psi^{\pm}_{k}$, we can assume without lose of generality that, by using a partition of unity, the support of $\psi^{\pm}_{k}$ is contained in a small relatively compact open set $U^{\pm}_{k}$ in $T^{1} M_k$, and there is a small relatively compact open set $\widetilde{U^{\pm}_{k}}$ in $T^{1}\HH^{n}$ such that the restriction of the quotient map $q_k: T^{1}\HH^{n}\rightarrow T^{1} M_{k}$ to $\widetilde{U^{\pm}_{k}}$ is a diffeomorphism to $U^{\pm}_{k}$. Define  $\widetilde{\psi^\pm_k}\in C^\infty_0(T^1\mathbb{H})$ with support in $\widetilde{U^{\pm}_{k}}$ and coinciding with  $\psi^\pm\circ q_k$ on $\widetilde{U^{\pm}_{k}}$. Similarly, we can define a compactly supported function $\tilde{\psi}\in C^\infty_0(T^1\mathbb{H})$ corresponding to $\psi$. Observe that we can choose the lifts $\widetilde{U^{\pm}_{k}}$ and $\widetilde{\psi^{\pm}_{k}}$ appropriately such that  $\widetilde{\psi^\pm_k}$ converges strongly to $\widetilde{\psi^\pm}$ and
\[\int_{\partial^1_\pm \widetilde{D^\mp_k}} \widetilde{\psi^\mp_k} d\sigma^\pm_{\widetilde{D^\mp_k}}  = \int_{\partial^1_\pm D^\mp_k} \psi^\mp_k d\sigma^\pm_{\partial D^\mp_k},
\]
where $\widetilde{D^{\pm}_k}$ are lifts of $D^{\pm}_k$. 
From now  on, we will distinguish components of $\widetilde{D^\pm_k}$. By  abuse of notation we still denote by $D^\pm_k$ a connected component of $\widetilde{D^\pm_k}$, which we assume is the only connected component of $\widetilde{D^{\pm}_{k}}$ so that the intersection of $\partial^1_\mp D^\pm_k$ with $\widetilde{U^{\pm}_{k}}$ is non-empty by using partition of unity. Observe that then we can label other components by $\Gamma_k$ left action $\gamma\mapsto\gamma D^\pm_k$. These labels are redundant (i.e. label the same set) if and only if $\gamma_1^{-1}\gamma_2$ belongs to the stabilizer of $D^\pm_k$. 

Take $\eta, R>0$ and $k$ sufficiently large so that the statement of Proposition \ref{prop:radiuscontinuity} applies for the sequence of convergent precompact subsets $(\Omega^\pm_k:=\partial^1_\mp \widetilde{D^\pm_k} \cap supp(\widetilde{\psi^\pm_k}))_{k\in\mathbb{N}}$. We will fix $R>0$ from now on, but will keep taking smaller (independent of $k$) $\eta$.  Observe that for small, fixed $\tau>0$ we have  the inclusion

\[\mathcal{V}_{\eta e^{-\tau},Re^{-\tau}} (\Omega^\pm_k) \subset \mathcal{V}_{\eta,R} (\Omega^\pm_k)
\]
is precompact in each slice $V^{\pm}_{w,\eta,R}$, and converges as a whole to $\mathcal{V}_{\eta e^{-\tau},Re^{-\tau}} (\Omega^\pm) \subset \mathcal{V}_{\eta,R} (\Omega^\pm)$ in the usual sense. If by $\mathbbm{1}_A$ we denote the characteristic function of a set $A$, then we can construct smooth functions $\chi^\pm_k\in C^\infty(T^1 \mathbb{H}^n)$ so that the following items hold

\begin{enumerate}
    \item For $w\in \Omega^\pm_k, v \in W^{0\mp}(w)$ $$\mathbbm{1}_{\mathcal{V}_{\eta e^{-\tau},Re^{-\tau}} (\Omega^\pm_k)}(v) \leq \chi^\pm_k(v) \leq \mathbbm{1}_{\mathcal{V}_{\eta ,R} (\Omega^\pm_k)}(v).$$
    \item The Sobolev norms $\Vert \chi^\pm_k \Vert_\beta$ are uniformly bounded (i.e. independent of $k$), where $\beta$ is the Sobolev norm appearing in the statement of \cite[Theorem 1.1]{EdwardsOh21}. 
    \item For any $w\in \Omega^\pm_k$ we have that
    \[e^{-\epsilon} \nu^\pm_w (V^\mp_{w,\eta,R})  \leq \int_{V^\mp_{w,\eta,R}} \chi^\pm_k d\nu^\pm_w \leq \nu^\pm_w (V^\mp_{w,\eta,R})
    \]
    for $\epsilon>0$ independent of $k$, where $d\nu^\pm_w := dsd\mu_{W^\mp(w)}$.
\end{enumerate}

In order to define the test functions to apply exponential mixing, we start with the functions $H^\pm_k:\partial^1_\mp\widetilde{ D^\pm_k} \rightarrow \mathbb{R}$ defined by

\[H^\pm_k(w) = \frac{1}{\int_{V^\mp_{w,\eta,R}} \chi^\pm_k d\nu^\pm_w}.
\]

Let $\Phi^\pm_k:T^1 \mathbb{H}^n\rightarrow \mathbb{R}$ defined by
\[\Phi^\pm_k = (H^\pm_k \widetilde{\psi^\pm_k})\circ f^\mp_{D^\pm_k}\chi^\pm_k.
\]
By construction, we have that $\Vert \Phi^\pm_k \Vert_\beta$ are uniformly bounded and have support in $\mathcal{V}_{\eta,R} (\Omega^\pm_k)$. Moreover, $\Phi^\pm_k$ are non-negative, measurable functions satisfying

\begin{equation}\label{eq:testfunctionintegral}
\int_{T^1\mathbb{H}^n}\Phi^\pm_k d\tilde{m}^k_{\rm{BM}} = \int_{\partial^1_\mp D^\pm_k} \psi^\pm_k d\sigma_{k}^\mp
\end{equation}

Following \cite{ParkkonenPaulin21}, we will estimate in two ways the quantity

\begin{equation}
    I_k(T) := \int^T_0 e^{\delta(\Gamma_k) t} \sum_{\gamma\in\Gamma_k} \int_{T^1\mathbb{H}^n} (\Phi^-_k\circ g^{-t/2})(\Phi^+_k\circ g^{t/2}\circ \gamma^{-1}) d\tilde{m}^k_{\rm{BM}}dt
\end{equation}

By \cite[Theorem 1.1]{EdwardsOh21} and Theorem \ref{thm:uniformgap} there exist uniform $\kappa>0, O(.)$ such that

\begin{equation}\label{eq:firstassympI}
\begin{split}
    I_k(T) &= \int^T_0 e^{\delta(\Gamma_k) t} \left(\frac{1}{\Vert m^k_{\rm{BM}}\Vert} \int_{T^1\mathbb{H}^n}\Phi^-_k d\tilde{m}^k_{\rm{BM}} \int_{T^1\mathbb{H}^n}\Phi^+_k d\tilde{m}^k_{\rm{BM}} + O(e^{-\kappa t}\Vert \Phi^-_k \Vert_\beta\Vert \Phi^+_k \Vert_\beta)\right)dt\\
    &=\frac{e^{\delta(\Gamma_k) T}}{\delta(\Gamma_k)\Vert m^k_{\rm{BM}}\Vert} \int_{\partial^1_- D^-_k} \psi^- d\sigma_{k}^-  \int_{\partial^1_+ D^+_k} \psi^+ d\sigma_{k}^+ + \int^T_{0} e^{\delta(\Gamma_k) t} O(e^{-\kappa t}\Vert \Phi^-_k \Vert_\beta\Vert \Phi^+_k \Vert_\beta)dt\\
    &=e^{\delta(\Gamma_k) T} \left( \frac{\sigma^{+}_{D^{-}_k}(\psi^-_k)\cdot  \sigma^{-}_{D^{+}_k}(\psi^+_k)}{\delta(\Gamma_k) \Vert m^k_{\rm{BM}}\Vert} + e^{-\delta(\Gamma_k) T}\int^T_{0} e^{\delta(\Gamma_k) t} O(e^{-\kappa t}\Vert \Phi^-_k \Vert_\beta\Vert \Phi^+_k \Vert_\beta)dt\right)
\end{split}
\end{equation}
where we used (\ref{eq:testfunctionintegral}) for the second equality. Observe in the final line that we can make the error term $e^{-\delta(\Gamma_k) T}\int^T_{0} e^{\delta(\Gamma_k) t} O(e^{-\kappa t}\Vert \Phi^-_k \Vert_\beta\Vert \Phi^+_k \Vert_\beta)dt$ arbitrarily small for any $T>T_0$, where $T_0$ sufficiently large and independent of $k$.

Now we use a second way to compute this integral $I_{k}(T)$. Let $\delta_k=\delta(\Ga_k)$. We interchange the integral over $t$ and the summation over $\gamma$. Then 
$$I_{k}(T)=\sum_{\gamma\in \Gamma_k} \int_{0}^{T} e^{\delta_{k} t} \int_{T^{1}\HH^{n}}(\Phi^{-}_{k} \circ g^{-t/2})(\Phi^{+}_{k}\circ g^{t/2}\circ \gamma^{-1})d \tilde{m}^{k}_{\rm{BM}} dt. $$
Suppose that if $v\in T^{1}\HH^{n}$ belongs to the support of $(\Phi^{-}_{k} \circ g^{-t/2})(\Phi^{+}_{k}\circ g^{t/2}\circ \gamma^{-1})$, then 
$$v\in g^{t/2} \V_{\eta, R}(\partial^{1}_{+}D_{k}^{-})\cap g^{-t/2} \V_{\eta, R}(\gamma \partial^{1}_{-}D_{k}^{+}).$$
Then by \cite[Lemma 7]{ParkkonenPaulin21}, which is proved by using hyperbolic geometry in $\HH^{n}$, we have the following 
\begin{equation}
\label{eq1}
d(w_{k}^{\pm}, v_{\gamma}^{\pm})=O(\eta+e^{-\ell_{\gamma}/2})
\end{equation}
where $w_{k}^{-}=f^{+}_{D_{k}}(v), w_{k}^{+}=f^{-}_{\gamma D_{k}^{+}}(v)$, $v_{\gamma}^{\pm}$ are endpoints of the common perpendicular between $D_{k}^{-}$ and $\gamma D_{k}^{+}$, and $\ell_{\gamma}$ is the length of the common perpendicular. Since the Lipschitz norm of $\widetilde{\psi^\pm_k}$ are uniformly bounded, and in particular bounded by the $\beta$ Sobolev norm of $\psi^\pm_k$, we have 
$$|\widetilde{\psi^{\pm}_{k}}(w^{\pm}_{k})-\widetilde{\psi^{\pm}_{k}}(v^{\pm}_{\gamma})|=O((\eta+e^{-\ell_{\gamma}/2})|| \psi_{k}^{\pm}||_{\beta}).$$
If we define $\hat{\Phi}^{\pm}_{k}=H^{\pm}_{k}\circ f^{\mp}_{D^{\pm}_{k}}\chi^{\pm}_{k}$ so that $\Phi^\pm_k= (\widetilde{\psi^\pm_k}\circ f^\mp_{D_{k}^{\pm}})\hat{\Phi}^\pm_k$, by applying the previous equation we obtain
\begin{equation}
\begin{aligned}
I_{k}(T)=\sum_{\gamma \in \Gamma_{k}} [\psi_{k}^{\pm}(v_{\gamma}^{-}) \psi_{k}^{\pm}(v^{+}_{\gamma})&+O((\eta+e^{-\ell_{\gamma}/2})|| \psi^{-}_{k}||_{\beta} || \psi^{+}_{k}||_{\beta})]\times \\
&\int_{0}^{T} e^{\delta_{k}t} \int_{v\in T^{1}\HH^{n}} \hat{\Phi}^{-}_{k}(g^{-t/2}v) \hat{\Phi}^{+}_{k}(\gamma^{-1} g^{t/2}v) d\tilde{m}_{\rm{BM}}^{k}(v)dt, 
\end{aligned}
\end{equation}
for $O(.)$ independent of $k$.

We now related another test function to  $\hat{\Phi}^{\pm}_{k}$ following \cite[Lemma 8]{ParkkonenPaulin21}. Let $h^{\pm}_{k}: T^{1}\HH^{n}\rightarrow \left[ 0, \infty \right] $ be the  $\Gamma_{k}$-invariant measurable map defined by
\begin{equation}
\label{def:h}
h^{\mp}_{k}(w)=\dfrac{1}{2\eta \mu_{W^{\pm}_{w}}(B^{\pm}(w, R))}
\end{equation}
if $\mu_{W^{\pm}(w)}(B^{\pm}(w, R))>0$, and $h^{\pm}_{k}(w)=0$ otherwise. 
We define the test function $\phi^{\mp}_{k}=\phi^{\mp}_{\eta, R, \Omega^{\pm}_{k}}: T^{1}\HH^{n}\rightarrow \left[ 0, \infty\right] $ by 
$$\phi^{\mp}_{k}=h^{\mp}_{k}\circ f^{\pm}_{D^{\mp}_{k}}\mathbbm{1}_{\V_{\eta, R}(\Omega^{\mp}_{k})}.$$
By the properties of $\chi_{k}^{\pm}$, we have
$$\phi^{\pm}_{\eta e^{-\tau}, R e^{-\tau}, \partial^{1}_{\mp} \Omega^{\pm}_{k}} e^{-\epsilon}\leq \hat{\Phi}^{\pm}_{k}\leq \phi^{\pm}_{k}. $$
Hence, it suffices to consider the integral 
\begin{equation}
\label{eq2}
\begin{split}
i_{k}(T)=\sum_{\gamma \in \Gamma_{k}} [\psi_{k}^{\pm}(v_{\gamma}^{-}) \psi_{k}^{\pm}(v^{+}_{\gamma})&+O((\eta+e^{-\ell_{\gamma}/2})|| \psi^{-}_{k}||_{\beta} || \psi^{+}_{k}||_{\beta})]\times \\ &\int^{T}_{0} e^{\delta_{k} t} \int_{T^{1} \HH^{n}} (\phi^{-}_{k}\circ g^{-t/2})(\phi^{+}_{k}\circ g^{t/2}\circ \gamma^{-1})d \tilde{m}^{k}_{\rm{BM}} dt.
\end{split}
\end{equation}
By the definition of $\phi_{k}^{\pm}$,  the right hand side of \eqref{eq2} is equal to 
\begin{equation}
\begin{split}
\sum_{\gamma\in \Gamma_{k}} &[\psi_{k}^{\pm}(v_{\gamma}^{-}) \psi_{k}^{\pm}(v^{+}_{\gamma})+O((\eta+e^{-\ell_{\gamma}/2})|| \psi^{-}_{k}||_{\beta} || \psi^{+}_{k}||_{\beta})]\times \\ &\int_{0}^{T} e^{\delta_{k}t} \int_{T^{1}\HH^{n}} h^{-}_{k}\circ f^{+}_{D_{k}^{-}}(g^{-t/2}v)h^{+}_{k}\circ f^{-}_{D_{k}^{+}}(\ga^{-1} g^{t/2}v)
\times \mathbbm{1}_{\V_{\eta, R}(\Omega_{k}^{-})}(g^{-t/2}v)\mathbbm{1}_{\V_{\eta, R}(\Omega^{+}_{k})}(\gamma^{-1}g^{t/2}v)d\tilde{m}_{\rm{BM}}^{k}dt.
\end{split}
\end{equation}
By the $\Gamma$-invariance of $h^{\pm}_{k}$, one has 
$$h^{-}_{k} \circ f^{+}_{D^{-}_{k}}(g^{-t/2}v)=e^{-\delta_{k}(t/2)} h^{-}_{k, e^{-t/2}R}(g^{t/2}w^{-}_{k}),$$
$$h^{+}_{k}\circ f^{-}_{D^{+}_{k}}(\gamma^{-1} g^{t/2}v)=e^{-\delta_{k}(t/2)}h^{+}_{k, e^{-t/2}R}(g^{-t/2}w^{+}_{k})$$
where $w^{-}_{k}=f^{+}_{D^{-}_{k}}(v)$,  $w^{+}_{k}=f^{-}_{\gamma D^{+}_{k}}(v)=\gamma f^{-}_{D^{+}_{k}} (\gamma^{-1}v)$, and $h^{-}_{k, e^{-t/2}R}$ is defined the same as in \eqref{def:h} except we  replace $R$ by $e^{-t/2}R$ . Therefore, 
$$h^{-}_{k}\circ f^{+}_{D_{k}^{-}}(g^{-t/2}v)h^{+}_{k}\circ f^{-}_{D_{k}^{+}}(\ga^{-1} g^{t/2}v)=e^{-\delta_{k}t} h^{-}_{k, e^{-t/2}R}(g^{t/2}w^{-}_{k})h^{+}_{k, e^{-t/2}R}(g^{-t/2}w^{+}_{k}).$$

The remaining part $\mathbbm{1}_{\V_{\eta, R}(\Omega_{k}^{-})}(g^{-t/2}v)\mathbbm{1}_{\V_{\eta, R}(\Omega^{+}_{k})}(\gamma^{-1}g^{t/2}v)$ if nonzero if and only if 
$$v\in g^{t/2} \V_{\eta, R}(\Omega_{k}^{-})\cap \gamma g^{-t/2} \V_{\eta, R}(\Omega_{k}^{+})=\V_{\eta, e^{-t/2}R}(g^{t/2}\Omega^{-}_{k})\cap \V_{\eta, e^{-t/2}R}(\gamma g^{-t/2}\Omega^{+}_{k}).$$

By \cite[Lemma 7]{ParkkonenPaulin21}, there exist constants $t_{0}>0$ and $c_{0}$ (independent of $k$), such that if $t\geq t_{0}$, the followings holds: there exists a common perpendicular $\alpha_{\gamma}$ from $D^{-}_{k}$ to $\gamma(D^{+}_{k})$ with 
\begin{enumerate}
\item $\vert \ell_{\gamma}-t \vert\leq 2\eta+c_{0} e^{-t/2}$,
\item $d(\pi(v^{\pm}_{\ga}), \pi(w^{\pm}_{k}))\leq c_{0}e^{-t/2},$
\item $d(\pi(g^{\pm t/2}w^{\mp}_{k}), \pi(v))\leq \eta+c_{0}e^{-t/2}$. 

\end{enumerate}

For all $\gamma\in \Gamma_{k}$ and $T\geq t_{0}$, we define
$$\A_{k, \gamma}(T)=\{ (t, v)\in [t_{0}, T]\times T^{1}\HH^{n}: v\in \V_{\eta, e^{-t/2}R}(g^{t/2}\Omega^{-}_{k})\cap \V_{\eta, e^{-t/2}R}(\gamma g^{-t/2}\Omega^{+}_{k})\},$$
and the integral
\begin{equation}
\begin{aligned}
j_{k, \gamma}(T)& =\int\int_{(t, v)\in \A_{k, \gamma}(T)} h^{-}_{k, e^{-t/2}R}(g^{t/2}w^{-}_{k})h^{+}_{k, e^{-t/2}R} (g^{-t/2} w^{+}_{k}) dt d \tilde{m}_{BM}^{k}(v)\\
&=\dfrac{1}{(2\eta)^{2}} \int \int_{(t, v)\in \A_{k, \gamma}(T)}\dfrac{dt d\tilde{m}^{k}_{BM}(v)}{\mu_{W^{+}(w^{-}_{t})}(B^{+}(w^{-}_{t}, r_{t})) \mu_{W^{-}(w^{+}_{t})} (B^{-}(w^{+}_{t}, r_{t}))}
\end{aligned}
\end{equation}
where 
$$r_{t}=e^{-t/2}R, \quad w^{-}_{t}=g^{t/2}w^{-}_{k}, \quad w^{+}_{t}=g^{-t/2} w^{+}_{k}. $$

There exists then a constant $c''_{0}>0$ (independent of $k$) such that for $T\geq t_{0}$, one has

\begin{equation}\label{eq:errorik}
\begin{split}
    -c''_{0}+\sum_{\ga\in \Ga_{T-O(\eta+e^{-\ell_{\ga}/2}), -\Op, k}} [\psi_{k}^{\pm}(v_{\gamma}^{-}) \psi_{k}^{\pm}(v^{+}_{\gamma})&+O((\eta+e^{-\ell_{\gamma}/2})|| \psi^{-}_{k}||_{\beta} || \psi^{+}_{k}||_{\beta})] j_{k, \ga}(T)\\&\leq i_{k}(T)\leq\\  c''_{0}+\sum_{\ga\in \Ga_{T+\Op, \Op, k}} [\psi_{k}^{\pm}(v_{\gamma}^{-}) \psi_{k}^{\pm}(v^{+}_{\gamma})&+O((\eta+e^{-\ell_{\gamma}/2})|| \psi^{-}_{k}||_{\beta} || \psi^{+}_{k}||_{\beta})] j_{k, \ga}(T+\Op),
\end{split}
\end{equation}
where $\Ga_{s, r, k}=\{\ga\in \Ga_k| t_0+2+c_0\leq \ell_\ga\leq s, v^{\pm}_{\ga}\in N_{r}\Omega^{\pm} \}$ for all $s, r\in \mathbb{R}$.

\begin{claim}\label{claim:jsize}
For any $\epsilon>0$, if $\eta$ is small enough and $\ell_{\gamma}$ is large enough, then 
$$j_{k, \gamma}(T)=e^{O(\eta+e^{-\ell_{\ga}/2})}e^{O(\epsilon^{c'})}\dfrac{(2\eta+O(e^{-\ell_{\ga}/2}))^{2}}{(2\eta)^{2}},$$
for $c'>0$ independent of $k$.
\end{claim}

\begin{proof}
Since $(\HH^{n}, \Gamma_{k})$ has radius-continous strong stable/unstable ball masses. By \cite[Lemma 11]{ParkkonenPaulin21}, for  every $\epsilon>0$ and every $(t, v)\in \A_{k, \ga}(T)$, one has 
$$\mu_{W^{\pm}(w_{t}^{\mp})}(B^{\pm}(w^{\mp}_{t}, r_{t}))=e^{O(\epsilon)} \mu_{W^{\pm}(v_{\gamma})} (B^{\pm}(v_{\ga}, r_{\ell_{\gamma}}))$$
if $\eta$ is small enough and $\ell_{\ga}$ is large enough, independent of $k$. Here $v_{\gamma}$ denote the midpoint of the common perpendicular from $D^{-}_{k}$ to $\ga(D^{+}_{k})$. Hence,
$$j_{k, \gamma}(T)=\dfrac{e^{O(\epsilon)}\int\int_{(t, v)\in \A_{k, \gamma}(T)} dt d\tilde{m}_{\rm{BM}}^{k}(v)}{\mu_{W^{+}(v_{\gamma})} (B^{+}(v_{\ga}, r_{\ell_{\gamma}}))\mu_{W^{-}(v_{\gamma})} (B^{-}(v_{\ga}, r_{\ell_{\gamma}}))}$$
By \cite[Lemma 10]{ParkkonenPaulin21},  for every $(t, v)\in \A_{k, \ga}(T)$, one has
$$dt d\tilde{m}_{\rm{BM}}^{k}(v)=e^{O(\eta+e^{-\ell_{\ga}/2})}dt ds d\mu_{W^{-}(v_{\ga})}(v') d\mu_{W^{+}(v_{\ga})}(v'')dt$$
where $v'=f^{+}_{HB_{-}(v_{\ga})}(v)$ and $v''=f^{-}_{HB_{+}(v_{\ga})}(v)$. By \cite[Lemma 9]{ParkkonenPaulin21}, the distances $d(v, v_{\ga}), d(v', v_{\ga})$ and $d(v'', v_{\ga})$ are $O(\eta+e^{-t/2})$. Combining these equations together, the claim follows. 
\end{proof}

Applying then Claim \ref{claim:jsize} in equation (\ref{eq:errorik}) we get

\begin{equation}\label{eq:2ndassympI}
\begin{split}
I_k(t) = \sum_{\gamma\in \Gamma_{k}} &[\psi_{k}^{\pm}(v_{\gamma}^{-}) \psi_{k}^{\pm}(v^{+}_{\gamma})] + O(\eta e^{\delta_k t}),
\end{split}
\end{equation}
for $t$ sufficiently large independent of $k$, and $O(.)$ independent of $k,\eta$. Then by multiplying $e^{-\delta_k t}$ to equations (\ref{eq:firstassympI}), (\ref{eq:2ndassympI}) we get that for fixed $\eta>0$

\begin{equation}
    \frac{\sigma^{+}_{D^{-}_k}(\psi^-_k)\cdot  \sigma^{-}_{D^{+}_k}(\psi^+_k)}{\delta(\Gamma_k) ||m^k_{\rm{BM}}||} - O(\eta) \leq \frac{N_{\psi^{-}_k, \psi^{+}_k}(t)}{e^{\delta(\Ga_k) t}} \leq \frac{\sigma^{+}_{D^{-}_k}(\psi^-_k)\cdot  \sigma^{-}_{D^{+}_k}(\psi^+_k)}{\delta(\Gamma_k) ||m^k_{\rm{BM}}||} + O(\eta)
\end{equation}
for $t$ sufficiently large independent of $k$ and $O(.)$ independent of $k,\eta$, from where the result follows.

\end{proof}

\medskip
\noindent
{\bf{More about the proof of Theorem \ref{mainthm:counting}}:}
As $(D^\pm_k)_{k\in\mathbb{N}}$ is a sequence of well-positioned convex sets in $M_k$ which converges strongly to a well-positioned set $D^\pm$ in $M$, we can select $\psi^\pm_k \in C^\infty_0(T^1M_k)$, $\psi^\pm\in C^\infty_0(T^1M)$ be compactly supported functions so that $(\psi^\pm_k)_{k\in\mathbb{N}}$ converges strongly to $\psi^\pm$ and $\psi^\pm_k \equiv 1$ in $supp(\sigma^\mp_{\partial D^\pm_k})$. Hence in the notation of Theorem \ref{thm:appendix}
\[\N_{\psi^-_k, \psi^+_k}(t)  =\N_{D^{-}_k, D^{+}_k}(t),\quad \sigma^\mp_{D^\pm_k}(\psi^\pm_k)= \Vert \sigma^\mp_{D^\pm_k}\Vert,
\]
and in particular $\frac{\sigma^{+}_{D^{-}_k}(\psi^-_k)\cdot  \sigma^{-}_{D^{+}_k}(\psi^+_k)}{\delta(\Gamma_k) ||m^k_{\rm{BM}}||}\neq 0$. Then we can restate the conclusion of Theorem \ref{mainthm:counting} by a multiplicative error uniformly close to 1 along the sequence as $t$ gets larger.

\qed

\bibliographystyle{amsalpha}
\bibliography{mybib}

\newcommand{\etalchar}[1]{$^{#1}$}
\providecommand{\bysame}{\leavevmode\hbox to3em{\hrulefill}\thinspace}
\providecommand{\MR}{\relax\ifhmode\unskip\space\fi MR }
\providecommand{\MRhref}[2]{%
  \href{http://www.ams.org/mathscinet-getitem?mr=#1}{#2}
}
\providecommand{\href}[2]{#2}
\begin{thebibliography}{ABB{\etalchar{+}}17}

\bibitem[ABB{\etalchar{+}}17]{7s}
Miklos Abert, Nicolas Bergeron, Ian Biringer, Tsachik Gelander, Nikolay
  Nikolav, Jean Raimbault, and Iddo Samet, \emph{{On the growth of
  $L^2$-invariants for sequences of lattices in Lie groups}}, Annals of
  Mathematics \textbf{185} (2017), no.~3, 711 -- 790.

\bibitem[Bab02]{Babillot}
Martine Babillot, \emph{On the mixing property for hyperbolic systems}, Israel
  J. Math. \textbf{129} (2002), 61--76. \MR{1910932}

\bibitem[BBB19]{BBB19}
Martin Bridgeman, Jeffrey Brock, and Kenneth Bromberg, \emph{Schwarzian
  derivatives, projective structures, and the {W}eil-{P}etersson gradient flow
  for renormalized volume}, Duke Math. J. \textbf{168} (2019), no.~5, 867--896.
  \MR{3934591}

\bibitem[BBP21]{BBVP21}
Martin Bridgeman, Kenneth Bromberg, and Franco~Vargas Pallete, \emph{The
  {W}eil-{P}etersson gradient flow of renormalized volume on a {B}ers slice has
  a global attracting fixed point}, 2021.

\bibitem[Bow93]{Bowditch93}
Brian~H. Bowditch, \emph{Geometrical finiteness for hyperbolic groups}, J.
  Funct. Anal. \textbf{113} (1993), no.~2, 245--317. \MR{1218098}

\bibitem[Bow95]{Bowditch95}
\bysame, \emph{Geometrical finiteness with variable negative curvature}, Duke
  Mathematical Journal \textbf{77} (1995), no.~1, 229--274.

\bibitem[CT99]{CanaryTaylor}
Richard~D. Canary and Edward~C. Taylor, \emph{Hausdorff dimension and limits of
  {K}leinian groups}, Geom. Funct. Anal. \textbf{9} (1999), no.~2, 283--297.
  \MR{1692482}

\bibitem[DOP00]{DalboOtalPeigne00}
Fran\c{c}oise Dal'bo, Jean-Pierre Otal, and Marc Peign\'{e}, \emph{S\'{e}ries
  de {P}oincar\'{e} des groupes g\'{e}om\'{e}triquement finis}, Israel J. Math.
  \textbf{118} (2000), 109--124. \MR{1776078}

\bibitem[EO21]{EdwardsOh21}
Sam Edwards and Hee Oh, \emph{Spectral gap and exponential mixing on
  geometrically finite hyperbolic manifolds}, Duke Math. J. \textbf{170}
  (2021), no.~15, 3417--3458. \MR{4324181}

\bibitem[Ham89]{Ursula89}
Ursula Hamenst\"{a}dt, \emph{A new description of the {B}owen-{M}argulis
  measure}, Ergodic Theory Dynam. Systems \textbf{9} (1989), no.~3, 455--464.
  \MR{1016663}

\bibitem[Ham04]{Ursula04}
\bysame, \emph{Small eigenvalues of geometrically finite manifolds}, J. Geom.
  Anal. \textbf{14} (2004), no.~2, 281--290. \MR{2051688}

\bibitem[LP82]{LaxPhillips}
Peter~D. Lax and Ralph~S. Phillips, \emph{The asymptotic distribution of
  lattice points in {E}uclidean and non-{E}uclidean spaces}, J. Functional
  Analysis \textbf{46} (1982), no.~3, 280--350. \MR{661875}

\bibitem[McM99]{McMullen99}
Curtis~T. McMullen, \emph{Hausdorff dimension and conformal dynamics. {I}.
  {S}trong convergence of {K}leinian groups}, J. Differential Geom. \textbf{51}
  (1999), no.~3, 471--515. \MR{1726737}

\bibitem[Pat76]{Patterson}
Samuel~J. Patterson, \emph{The limit set of a {F}uchsian group}, Acta Math.
  \textbf{136} (1976), no.~3-4, 241--273. \MR{450547}

\bibitem[PP17]{ParkkonenPaulin21}
Jouni Parkkonen and Fr\'{e}d\'{e}ric Paulin, \emph{Counting common
  perpendicular arcs in negative curvature}, Ergodic Theory Dynam. Systems
  \textbf{37} (2017), no.~3, 900--938. \MR{3628925}

\bibitem[Rob00]{Roblin00}
Thomas Roblin, \emph{Sur l'ergodicit\'{e} rationnelle et les propri\'{e}t\'{e}s
  ergodiques du flot g\'{e}od\'{e}sique dans les vari\'{e}t\'{e}s
  hyperboliques}, Ergodic Theory Dynam. Systems \textbf{20} (2000), no.~6,
  1785--1819. \MR{1804958}

\bibitem[Rob03]{Roblin03}
\bysame, \emph{Ergodicit\'{e} et \'{e}quidistribution en courbure
  n\'{e}gative}, M\'{e}m. Soc. Math. Fr. (N.S.) (2003), no.~95, vi+96.
  \MR{2057305}

\bibitem[Sul79]{Sullivan1}
Dennis Sullivan, \emph{The density at infinity of a discrete group of
  hyperbolic motions}, Inst. Hautes \'{E}tudes Sci. Publ. Math. (1979), no.~50,
  171--202. \MR{556586}

\bibitem[Sul84]{Sullivan2}
\bysame, \emph{Entropy, {H}ausdorff measures old and new, and limit sets of
  geometrically finite {K}leinian groups}, Acta Math. \textbf{153} (1984),
  no.~3-4, 259--277. \MR{766265}

\bibitem[Sul87]{Sullivan3}
\bysame, \emph{Related aspects of positivity in {R}iemannian geometry}, J.
  Differential Geom. \textbf{25} (1987), no.~3, 327--351. \MR{882827}

\end{thebibliography}

\end{document}